\documentclass[twoside,11pt]{article}

\usepackage{mathtools}
\usepackage{latexsym}
\usepackage{mathrsfs}
\usepackage{amssymb,bm}
\usepackage{amsmath}
\usepackage{amsfonts}
\usepackage{amsthm}
\usepackage{enumitem}
\usepackage{hyperref}
\usepackage{epigraph}

\hypersetup{
  unicode=true,
  colorlinks=true,
  linkcolor=blue,
  citecolor=blue,
  pdftitle={The Complex Spectral Flow and Two-Parameter Equivariant Bifurcation},
}

\providecommand{\U}[1]{\protect\rule{.1in}{.1in}}
\def\bc{{\mathbb{C}}}

\def\bn{{\mathbb{N}}}
\def\br{{\mathbb{R}}}

\def\bz{{\mathbb{Z}}}
\def\bt{{\mathbb{T}}}

\def\gdeg{G\text{\rm -deg}}
\def\vs{\vskip.3cm}
\def\noi{\noindent}
\DeclareMathOperator{\id}{Id}
\def\s1deg{S^1\text{\rm -deg}}
\def\Om{\Omega}
\def\Gammadeg{\Gamma \text{\rm -deg}}

\usepackage{xcolor}
\usepackage{listings}

\definecolor{codegray}{rgb}{0.5,0.5,0.5}
\definecolor{Magenta}{rgb}{1,  0,  1}
\definecolor{codegray}{rgb}{0.5,0.5,0.5}
\definecolor{RoyalPurple}{rgb}{.25,  .10,  1}
\definecolor{NavyBlue}{rgb}{.06,  .46,  1}
\definecolor{cadmiumgreen}{rgb}{0.0, 0.42, 0.24}
\definecolor{Black}{rgb}{0,  0,  0}

\lstdefinestyle{mystyle}{
    commentstyle=\color{codegreen},
    keywordstyle=\color{magenta},
    numberstyle=\tiny\color{codegray},
    stringstyle=\color{RoyalPurple},
    basicstyle=\ttfamily\footnotesize,
    breakatwhitespace=false,         
    breaklines=true,                 
    captionpos=b,                    
    keepspaces=true,                 
    numbers=left,                    
    numbersep=5pt,                  
    showspaces=false,                
    showstringspaces=false,
    showtabs=false,                  
    tabsize=2
}
  \lstdefinelanguage{GAP}{
    basicstyle=\ttfamily,
    keywords={true, false, function, return, fail, if, in, while, do, od, else, elif, fi, break, continue},
    keywordstyle=\color{NavyBlue}\bfseries,
    otherkeywords={
      >, <, ==
    },
    breaklines=true,      
    identifierstyle=\color{Black},
    sensitive=True,
    comment=[l]{\#},
    commentstyle=\color{cadmiumgreen},
    stringstyle=\color{black},
    morestring=[b]',
    morestring=[b]"
  }
\newtheorem{theorem}{Theorem}[section]

\newtheorem{lemma}[theorem]{Lemma}
\newtheorem{corollary}[theorem]{Corollary}
\newtheorem{definition}[theorem]{Definition}
\newtheorem{remark}[theorem]{Remark}

\begin{document}

\title{The Complex Spectral Flow: Spectral Conditions for Two-Parameter Equivariant Bifurcation Guarantees} 

\author{
Ziad Ghanem\thanks{\small Department of Mathematics, Brandeis University, Waltham, MA 02453, USA
}}
\date{}

\maketitle

\begin{abstract}
We introduce the \emph{complex equivariant spectral flow}---a virtual $G$-representation assembling the eigenvalue winding numbers of the linearization at an isolated two-parameter 
critical point for $G = S^1 \times \Gamma$ bifurcation problems---and prove that, for maximal twisted orbit types, the coefficient of the local bifurcation invariant in $A_1^t(G)$ reduces to a closed-form dimension formula, 
bypassing the standard pipeline of basic degree factorization and Burnside ring multiplication entirely.
When the eigenvalue dependence is holomorphic, topological cancellation among winding numbers is impossible, yielding unconditional local and global bifurcation guarantees.
As applications, we establish the macroscopic escape of symmetric Hopf branches in $\Gamma$-equivariant systems and of patterned relative equilibria for the complex Ginzburg--Landau equation directly from spectral data.
\end{abstract}

\noi \textbf{Mathematics Subject Classification:} Primary: 37G40, 47H11; Secondary: 34C25, 55M25

\medskip

\noi \textbf{Key Words and Phrases:} equivariant degree, spectral flow, two-parameter bifurcation, complex representations, generalized kernel, winding number.

\setlength{\epigraphwidth}{0.54\textwidth}
\epigraph{Among the words which have had this happy result I will mention the group and the invariant.}{Henri Poincaré}

\section{Introduction} \label{sec:intro}

Bifurcations are, at their core, spectral events.  Given a one-parameter family of equations $\mathscr{F}(\alpha,u) = 0$ admitting a trivial solution branch $u \equiv 0$, the emergence of nontrivial solutions is governed by the spectrum of the linearization $\mathscr{A}(\alpha) := D_u\mathscr{F}(\alpha,0)$.  In the case that $\mathscr{A}(\alpha)$ is a compact perturbation of the identity, the Leray--Schauder degree \cite{LeraySchauder} reduces the detection of eigenvalue crossings to a parity count: the degree of $\mathscr{A}(\alpha)$ on a small ball around the origin equals $(-1)^{n_-(\alpha)}$, where $n_-(\alpha)$ counts the negative eigenvalues. Krasnosel'skii's bifurcation theorem~\cite{Krasnoselskii} guarantees the emergence of a nontrivial solution branch whenever $n_-$ changes parity across an isolated critical value. Global continuation is then provided by the Rabinowitz alternative~\cite{Rabinowitz}: any bifurcating branch either extends to infinity or returns to the trivial branch at another critical value.

\vs

When the governing equations are equivariant under a compact Lie group $G$, the negative eigenspace of the linearization $E^-(\alpha)$ inherits the $G$-action and decomposes into irreducible $G$-representations. The \emph{$G$-equivariant primary degree} of Geba, Krawcewicz, and Wu \cite{Geba1994} (see also \cite{AED, book-new, survey}) promotes the Leray--Schauder degree from a single integer to an element of the Burnside ring $A(G) := \bz[\Phi_0(G)]$, whose generators are indexed by the orbit types $(H) \in \Phi_0(G)$ with finite Weyl groups. The coefficient of each generator $(H)$ in the $G$-equivariant degree of the linearization is computed through a recurrence formula involving $(-1)^{\dim \mathcal{V}_i^H}$ for each irreducible $G$-representation $\mathcal{V}_i$ appearing in the negative eigenspace, and is assembled from  the corresponding \emph{basic degrees} $\deg_{\mathcal{V}_i} := G\text{-deg}(-\id, B(\mathcal{V}_i)) \in A(G)$. This framework resolves solution symmetries for all orbit types in $\Phi_0(G)$, providing substantial information when $G$ has a rich lattice of subgroups with finite Weyl groups---as is the case for finite groups, orthogonal groups, and their direct products. Orbit types with positive-dimensional Weyl groups, however, are invisible to the equivariant primary degrees. For example, in the case of $G = S^1$, one has $\Phi_0(S^1) = \{(S^1)\}$ and $A(S^1) \cong \bz$ such that the $S^1$-equivariant primary degree reduces to the classical Leray--Schauder degree, providing no resolution of the cyclic orbit types $(\bz_k)$ that characterize branches with distinct symmetries. 
\vs
The passage to \emph{two-parameter} bifurcation fundamentally changes both the underlying topology and the spectral datum that governs bifurcation. With one parameter, the boundary of an isolating interval is a pair of points, and the bifurcation invariant detects a \emph{sign change} between them. With two parameters, the boundary of an isolating disk $B_\varepsilon(\lambda_0) \subset \br^2$ is a circle, and the relevant invariant is a \emph{winding number}: the number of times an eigenvalue $\mu(\lambda)$ winds around the origin as $\lambda$ traces the boundary circle. The topological shift has a natural spectral counterpart. The datum governing two-parameter bifurcation is no longer the negative eigenspace, but the \emph{generalized kernel} of the linearization at the critical point, and the topological contributions are measured by algebraic multiplicities rather than eigenvalue parities. The natural degree-theoretic tool for this setting is the $S^1$-equivariant degree, introduced by Dylawerski, G\c{e}ba, Jodel, and Marzantowicz~\cite{Dylawerski1991, Dylawerski} and studied independently by Ize, Massab\`o, 
and Vignoli~\cite{Ize1989, Ize1992} (see also the Fuller index~\cite{Fuller1967}). This degree takes values in the free $\bz$-module $A_1(S^1) := \bz[\{(\bz_k)\}_{k \geq 1}]$, 
whose generators are the cyclic orbit types 
not contained in the corresponding Burnside ring.

\vs
In applications to symmetric Hopf bifurcation, periodic solutions of equivariant systems, and nonlinear wave equations, the symmetry group 
often has the product form $G = S^1 \times \Gamma$. The $S^1$ factor may arise for different reasons depending on the problem: as the temporal phase shift of periodic orbits (as in Section~\ref{sec:application}), as the gauge symmetry of a complex amplitude equation (as in Section~\ref{sec:CGL_application}), or as a rotational symmetry of the spatial domain; $\Gamma$ typically encodes the remaining spatial or network symmetries. The \emph{twisted $G$-equivariant degree} of Balanov, Krawcewicz, and Steinlein \cite{AED, book-new, survey} evaluates the topological charge of these problems, taking values in the $A(\Gamma)$-module $A_1^t(G)$ generated by the twisted orbit types. Historically, computing the twisted equivariant degree has relied on factorizing the degree into a Burnside ring component (for the stationary spectrum) and a twisted module component (for the non-stationary part), proceeding through a pipeline of basic degrees, the Splitting Lemma, and $A(\Gamma)$-module multiplication~\cite{AED, book-new, survey}. This framework has been applied extensively to symmetric Hopf bifurcation and nonlinear wave equations \cite{Erbe1992, KrawcewiczWuBook, FuraRatajczakRybicki,BalanovKrawcewiczRuan2006, GaGhKr}. The present paper shows that for maximal orbit types, this entire algebraic pipeline is unnecessary: the coefficients of 
the local bifurcation invariant can be read off directly from the spectral data of the linearization via the complex equivariant spectral flow.
\vs
\noi\textit{The two-parameter setting and the complex spectral flow.}\;
Two-parameter bifurcation problems with $S^1$-symmetries admit a natural complexification. Consider a two-parameter family of completely continuous $G$-equivariant fields $\mathscr{F}: \br^2 \times \mathscr{H} \to \mathscr{H}$, where $G = S^1 \times \Gamma$ and $\mathscr{H}$ is an isometric Banach $G$-representation. For these product groups, every irreducible representation $\mathcal V_{m,j}$ is uniquely identified by an $S^1$-isotypic index $m \in \bn \cup \{0\}$ and a $\Gamma$-isotypic index $j \in \mathcal J$. Accordingly, the phase space admits the $G$-isotypic decomposition
\[
\mathscr{H} = \overline{\bigoplus_{m=0}^\infty \bigoplus_{j \in \mathcal J} \mathscr{H}_{m,j}},
\]
where each $G$-isotypic component $\mathscr{H}_{m,j}$ is modeled on a direct sum of copies of the irreducible $G$-representation $\mathcal V_{m,j}$. Identifying the parameter plane $\br^2$ with $\bc$ via $\lambda = \alpha + i\beta$, Schur's Lemma ensures that the linearization $\mathscr A(\lambda):=D_u \mathscr F(\lambda,0)$ decomposes block-diagonally as $\mathscr A(\lambda) = \bigoplus_{m,j} \mathscr{A}_{m,j}(\lambda)$. The blocks with $m=0$ constitute the \emph{stationary} linearization, while the modes with $m \ge 1$ are \emph{dynamic}. On each dynamic component, $S^1$ acts non-trivially, so the eigenvalues of the corresponding restriction $\sigma(\mathscr{A}_{m,j}(\lambda)) =\{\mu_{m,j,k}(\lambda)\}_k$ are complex-valued functions of $\lambda$.

\vs
A trivial solution $(\lambda_0,0) \in \bc \times \mathscr H$ is called a \emph{dynamic critical point} if the stationary blocks ($m = 0$) of the linearization remain invertible at $\lambda_0$ while a finite number of dynamic eigenvalues (those in the $m \geq 1$ 
blocks) vanish.  As $\lambda$ traces a small circle $\partial B_\varepsilon(\lambda_0)$ around the singularity, each critical eigenvalue traces a closed curve whose winding number around the origin defines a local Brouwer degree. We define the \emph{complex equivariant spectral flow representation} as a virtual $G$-representation in $RO(G)$:
\begin{equation*}
\mathbb{W}_{\lambda_0} := \sum_{m \geq 1} \sum_{j \in \mathcal{J}} \rho_{m,j}(\lambda_0)\, \mathcal{V}_{m,j},
\end{equation*}
where $\rho_{m,j}(\lambda_0) := \sum_k \deg(\mu_{m,j,k}, B_\varepsilon(\lambda_0))$ is the total winding number of the critical eigenvalues in the $(m,j)$-th $G$-isotypic component. Since eigenvalues with opposite winding numbers can cancel, the total winding number $\rho_{m,j}(\lambda_0)$ may be negative, so $\mathbb{W}_{\lambda_0}$ is properly a virtual representation.

\vs

\noi\textit{The main bifurcation results.}\;
At an isolated, dynamic critical point $(\lambda_0,0) \in \bc \times \mathscr H$, the stationary linearization $\bigoplus_{j} \mathscr{A}_{0,j}(\lambda_0)$ is an isomorphism. For \textbf{maximal} orbit types $(H)$ in the support of $\mathbb{W}_{\lambda_0}$, the recurrence formula for the twisted equivariant degree collapses and the coefficient in the local bifurcation invariant reduces to the closed-form expression
\begin{equation}\label{eq:universal_intro}
\operatorname{coeff}^H(\omega_G(\lambda_0)) = \operatorname{sign}\det\left(\bigoplus_{j \in \mathcal J} \mathscr{A}_{0,j}^H(\lambda_0)\right) \cdot \frac{\dim_\bc \mathbb{W}_{\lambda_0}^H}{|W(H)/S^1|},
\end{equation}
where $\dim_\bc \mathbb{W}_{\lambda_0}^H := \sum_{m=1}^\infty \sum_{j \in \mathcal J} \rho_{m,j}(\lambda_0) \dim_\bc \mathcal{V}_{m,j}^H$ is the \emph{formal complex dimension} of the $H$-fixed subspace. This coefficient is nonzero whenever the formal dimension is nonzero (Theorem~\ref{thm:generalized_guarantee}). The proof exploits maximality to collapse the recurrence formula for the twisted equivariant degree, identifying the cumulative $S^1$-degree directly with the formal dimension. No basic degree computation or Burnside ring multiplication is performed. 
\vs
\noindent\textit{Global bifurcation.}\;
Our local results extend to global continuation through a three-layer argument developed in 
Section~\ref{sec:unbounded_branches}. At the first layer, the equivariant Rabinowitz alternative (Theorem~\ref{thm:rab_alt}) guarantees that a bifurcating branch must exit any open, bounded $G$-invariant domain $\mathcal U \subset \bc \times \mathscr H$ whose boundary avoids the critical set, provided the aggregate invariant $\sum_i \omega_G(\lambda_i)$ is non-zero.
At the second layer, we translate this topological condition into spectral data by assembling the local spectral flows over $\mathcal U$ and weighting them by the stationary determinant signs on fixed-point spaces. For each twisted orbit type $(K)$ this induces a $K$-signed aggregate spectral flow $\mathbb W_{\mathcal U}(K)
:= \sum_{i} \operatorname{sign}\det_\br
\left( \bigoplus_{q\in\mathcal J}
\mathscr A_{0,q}^K(\lambda_i)
\right) \mathbb W_{\lambda_i}$. If $(H)$ is maximal in the representation formed as the direct sum of the supports of the local flow representations $\mathbb W_{\lambda_i}$ contained in $\mathcal U$ and the $H$-signed formal dimension $\dim_\bc \mathbb W_{\mathcal U}(H)^H$ is non-zero, then the aggregate invariant has a non-zero $(H)$-coefficient and macroscopic escape is guaranteed (Theorem~\ref{thm:macroscopic_escape}).
At the third layer, a spatio-temporal fixed-point reduction eliminates the stationary modes entirely, yielding the cleanest formulation: if $(H)$ is a maximal orbit type in the support of the \emph{aggregate virtual spectral flow} 
$\mathbb W_{\mathcal U} := \sum_i \mathbb W_{\lambda_i}$ and $\dim_\bc (\mathbb W_{\mathcal U})^H \neq 0$, then the $H$-symmetric branch is forced to exit $\mathcal U$ 
(Theorem~\ref{thm:spatio_temporal_macroscopic_escape}).
\vs
\noi\textit{Holomorphic specialization.}\;
For general (non-holomorphic) systems, the formal dimension can vanish due to cancellations among eigenvalues with opposing winding directions. However, when the eigenvalues depend holomorphically on $\lambda$, each winding number coincides with the order of vanishing of the corresponding eigenvalue at $\lambda_0$, which is strictly 
positive by the Argument Principle. Topological cancellation is therefore impossible, and the formal dimension is \emph{unconditionally} positive for every maximal orbit type 
in the generalized kernel. Consequently, the formula \eqref{eq:universal_intro} detects every maximal orbit type and guarantees the macroscopic escape of the associated solution branches (Corollary~\ref{cor:holomorphic_escape}; in the 
spatio-temporal setting, this simplifies further to Corollary~\ref{cor:spatio_temporal_holomorphic_escape}).
\vs

\noi\textit{Relationship to the existing literature.}\;
The present paper does not develop a new equivariant degree theory; the twisted $G$-equivariant degree of~\cite{AED, book-new, survey} and its axiomatic properties are used throughout. What is new is the identification of the complex spectral flow representation as the natural organizing object for two-parameter equivariant bifurcation, and the observation 
that the recurrence formula collapses to the explicit dimension formula~\eqref{eq:universal_intro} at maximal orbit types. The resulting coefficients agree with those produced by the standard basic degree pipeline, but the computation is reduced to a single spectral datum.
\vs
To the best of our knowledge, assembling complex eigenvalue winding numbers into a virtual group representation to evaluate 
local bifurcation invariants directly has not appeared in the bifurcation literature. The closest antecedents are the classical spectral flow of Atiyah, Patodi, and Singer~\cite{AtiyahPatodiSinger}, which 
assigns an integer to a one-parameter path of self-adjoint Fredholm operators by counting signed eigenvalue crossings through zero, and its $G$-equivariant generalization by Izydorek, Janczewska, and Waterstraat~\cite{Izydorek2021}, which promotes this integer to an element of $RO(G)$ and applies it to detect periodic solutions in $G$-equivariant Hamiltonian systems. Both frameworks, however, are intrinsically one-parameter and variational, tracking \emph{signed} crossings of real eigenvalues along a path, with cancellations between positive and negative crossings built into the formalism. By contrast, the complex spectral flow introduced here governs 
two-parameter, non-variational families, and the absence of topological cancellation in holomorphic systems provides unconditional bifurcation guarantees that have no analogue in 
the one-parameter setting.
\vs
\noi\textit{Organization.}\;
Section~\ref{sec:bif_framework} recalls the equivariant bifurcation framework for the group $G = S^1 \times \Gamma$ and defines the local bifurcation invariant via the twisted equivariant degree.
Section~\ref{sec:spectral_flow} introduces the complex spectral flow representation and establishes the main coefficient formula for maximal orbit types (Theorem~\ref{thm:generalized_guarantee}) together with its holomorphic specialization 
(Corollary~\ref{cor:holomorphic_universal_guarantee}). We then develop the global theory: a spectral criterion for macroscopic escape (Theorem~\ref{thm:macroscopic_escape}), its 
holomorphic strengthening (Corollary~\ref{cor:holomorphic_escape}), and a 
spatio-temporal fixed-point reduction yielding a pure dimension criterion for the global bifurcation of genuinely time-dependent orbits 
(Theorem~\ref{thm:spatio_temporal_macroscopic_escape}).
Section~\ref{sec:application} applies the framework to symmetric Hopf bifurcation in a parameterized $\Gamma$-equivariant ODE, comparing the standard basic-degree pipeline with the spectral flow approach.
Section~\ref{sec:CGL_application} applies the holomorphic specialization to the complex Ginzburg--Landau equation on compact manifolds, establishing universal macroscopic escape 
guarantees and recovering the spiral wave results of~\cite{GGK_NODEA} without explicit Jacobian verification.
Appendix~\ref{sec:appendix_S1} provides a self-contained account of the $S^1$-equivariant degree, and Appendix~\ref{sec:appendix_twisted_deg} develops the twisted equivariant degree and its computational formulae.

\section{The Equivariant Bifurcation Framework} \label{sec:bif_framework}
Throughout this paper, let 
$\Gamma$ be a compact Lie group and put $G:= S^1 \times \Gamma$. The irreducible representations of $S^1$ are the trivial representation $\mathcal U_0 \simeq \br$ and for each $m \in \bn$ the complex representation $\mathcal U_m \simeq \bc$ equipped with the $m$-folded $S^1$-action
\[
\varphi z := e^{im \varphi} \cdot z, \quad \varphi \in S^1, \; z \in \mathcal U_m.
\]
Given a complete list of the irreducible $\Gamma$-representations $\operatorname{Irr}(\Gamma) = \{ \mathcal V_j\}_{j \in \mathcal J}$, the irreducible $G$-representations can be uniquely specified (up to equivalence) with a $G$-isotypic index $(m,j) \in \bn \cup \{0\} \times \mathcal J$ using the notation
\[
\mathcal V_{m,j} := \mathcal U_m \otimes \mathcal V_j.
\]
Let $\mathscr H$ be an isometric Banach $G$-representation with the $G$-isotypic decomposition
\begin{equation}\label{eq:isotypic_decomp_H}
    \mathscr{H} = \overline{\bigoplus_{m=0}^\infty \bigoplus_{j \in \mathcal{J}} \mathscr{H}_{m,j}}, \quad \mathscr H_{m,j} \simeq \ell_{m,j} \mathcal V_{m,j},
\end{equation}
and consider the two-parameter bifurcation problem 
\begin{align}\label{eq:bif_problem}
    \mathscr{F}(\alpha,\beta, u) = 0, \quad (\alpha,\beta, u) \in \br \times \br \times \mathscr{H},
\end{align}
where $\mathscr{F}: \br \times \br \times \mathscr{H} \to \mathscr{H}$ is a completely continuous $G$-equivariant field satisfying $\mathscr{F}(\alpha,\beta, 0) = 0$ for all $(\alpha,\beta) \in \br \times \br$ and is continuously Fréchet differentiable in a neighborhood of the trivial branch, admitting a linearization $\mathscr{A}(\alpha,\beta) := D_u \mathscr{F}(\alpha,\beta, 0)$ that depends continuously on the bifurcation parameters $(\alpha,\beta) \in \br \times \br$.
\vs
The set of all solutions to \eqref{eq:bif_problem} can be divided into the set of {\it trivial solutions} $M := \{ (\alpha,\beta,0) \in \br \times \br \times \mathscr H \}$
and the set of {\it non-trivial solutions}
\[
\mathscr S := \{ (\alpha,\beta,u) \in \br \times \br \times \mathscr H: \mathscr  F(\alpha,\beta,u) = 0, \; u \neq 0 \}.
\]
Given the $G$-symmetry of our bifurcation problem, we are particularly interested in non-trivial solutions with prescribed symmetries. For any closed subgroup $H \leq G$, we can always consider the $H$-fixed-point set $\mathscr S^H := \{ (\alpha,\beta,u) \in \mathscr S : G_u \geq H \}$ consisting of all non-trivial solutions to \eqref{eq:bif_problem} with {\it symmetries at least $(H) \in \Phi(G;\mathscr H)$}.
\vs
For simplicity of notation, we identify $\br \times \br$ with the complex plane $\bc$ by associating each pair of parameters $(\alpha,\beta) \in \br \times \br$ with the complex number $\lambda := \alpha + i \beta$. A trivial solution $(\lambda_0,0) \in M$ is said to be a {\it bifurcation point} for \eqref{eq:bif_problem} if every open neighborhood of the point $(\lambda_0,0)$ has a non-empty intersection with the set of non-trivial solutions $\mathscr S$. A critical point $(\lambda_0,0) \in M$ is a trivial solution at which $\mathscr A(\lambda_0): \mathscr H \rightarrow \mathscr H$ is {\it not} an isomorphism, and the {\it critical set} for \eqref{eq:bif_problem} is 
\begin{align*} 
    \Lambda := \{ (\lambda,0) \in \br^2 \times \mathscr H : \mathscr A(\lambda): \mathscr H \rightarrow \mathscr H \text{ is not an isomorphism}\}.   
\end{align*}
A critical point $(\lambda_0,0) \in \Lambda$ is said to be \textit{isolated} if there exists $\varepsilon > 0$ with $(\overline{B_{\varepsilon} (\lambda_0)} \times \{0\})\cap \Lambda = \{ (\lambda_0,0) \}$. An isolated critical point $(\lambda_0,0) \in \Lambda$ is said to be a {\it branching point} for \eqref{eq:bif_problem} if there exists a non-trivial continuum $K \subset \overline{ \mathscr S}$ with $K \cap M = \{ (\lambda_0,0) \}$ and the maximal connected component $\mathscr C \subset \overline{\mathscr S}$ containing $K$ is called the \textit{branch} of nontrivial solutions with branching point $(\lambda_0,0)$. 
\vs
Now, let $(\lambda_0,0) \in \Lambda$ be an isolated critical point with a deleted $\varepsilon$-neighborhood $B_{\varepsilon}(\lambda_0) \subset \bc$ on which $\mathscr A(\lambda): \mathscr H \rightarrow \mathscr H$ is an isomorphism, and choose $\delta > 0$ sufficiently small such that the equation $\mathscr{F}(\lambda, u) = 0$ admits no non-trivial solutions on the boundary of the isolating cylinder
\begin{equation*}
    \mathscr{O} := \{ (\lambda, u) \in \bc \times \mathscr{H} : |\lambda - \lambda_0| < \varepsilon, \; \|u\| < \delta \}.
\end{equation*}
A $G$-invariant function $\Theta: \bc \times \mathscr H \rightarrow \br$ is said to be an {\it auxiliary function} on $\mathscr O$ if it satisfies
\begin{align*} 
   \begin{cases}
\Theta(\lambda,0) < 0 \quad & \text{ for } |\lambda - \lambda_0 | = \varepsilon; \\
\Theta(\lambda,u) > 0  \quad & \text{ for } |\lambda - \lambda_0 | \leq \varepsilon \text{ and } \Vert u \Vert_{\mathscr H} = \delta.
\end{cases}    
\end{align*}
Given any auxiliary function $\Theta$, the {\it complemented operator}
\begin{align*} 
\mathscr F_\Theta : \bc \times \mathscr H \rightarrow \br \times \mathscr H, \quad   \mathscr F_\Theta(\lambda,u):= (\Theta(\lambda,u),\mathscr F(\lambda,u)),
\end{align*}
is an $\mathscr O$-admissible $G$-map. Since $\mathscr{F}$ is Fréchet differentiable at the origin with a derivative that depends continuously on $\lambda$, its nonlinear remainder is asymptotically negligible for small amplitudes, allowing the complemented operator $\mathscr F_\Theta$ to be $\mathscr O$-admissibly $G$-homotoped to the \emph{complemented linear operator} $\mathscr A_\Theta(\lambda,u):= (\Theta(\lambda,u),\mathscr A(\lambda,u))$.
\begin{lemma}\label{lem:linearization_homotopy}
Let $(\lambda_0,0) \in \Lambda$ be an isolated critical point and let $\Theta$ be an auxiliary function on $\mathscr{O}$. For $\varepsilon > 0$ and $\delta > 0$ sufficiently small, the complemented operator $\mathscr{F}_\Theta$ is $\mathscr{O}$-admissibly $G$-homotopic to the complemented linear operator 
\[
\mathscr{A}_\Theta(\lambda,u) := (\Theta(\lambda,u), \mathscr{A}(\lambda)u).
\]
\end{lemma}
\begin{proof}
By the continuous $u$-Fréchet differentiability of $\mathscr F$ near the trivial branch, the map
$(\lambda,u)\longmapsto D_u\mathscr F(\lambda,u)$ is continuous. Hence, after fixing $\varepsilon>0$ sufficiently small, the compactness of $\overline{B_\varepsilon(\lambda_0)}\times\{0\}$ implies that for every $\eta>0$ there exists $\delta>0$ such that $\|D_u\mathscr F(\lambda,tu)-D_u\mathscr F(\lambda,0)\|<\eta$
for all $\lambda\in\overline{B_\varepsilon(\lambda_0)}$, $\|u\|\leq\delta$, and $t\in[0,1]$. Therefore, the integral remainder formula
\[
\mathscr R(\lambda,u)
:=\mathscr F(\lambda,u)-\mathscr A(\lambda)u
=
\int_0^1
\left[D_u\mathscr F(\lambda,tu)-D_u\mathscr F(\lambda,0)\right]u\,dt,
\]
implies $\|\mathscr R(\lambda,u)\|\leq \eta\|u\|$, i.e. $\mathscr R(\lambda,u)=o(\|u\|)$ uniformly for $\lambda\in\overline{B_\varepsilon(\lambda_0)}$.

We construct the straight-line $G$-homotopy $h_t(\lambda, u) := (\Theta(\lambda, u), \mathscr{A}(\lambda)u + t\mathscr{R}(\lambda,u))$ for $t \in [0,1]$. To prove $h_t$ is $\mathscr{O}$-admissible, we must show $h_t(\lambda,u) \neq 0$ on the boundary $\partial\mathscr{O}$, which decomposes into the lateral boundary $\overline{B_\varepsilon(\lambda_0)} \times \partial B_\delta(0)$ and the base boundary $\partial B_\varepsilon(\lambda_0) \times \overline{B_\delta(0)}$.

On the lateral boundary, one has $\|u\| = \delta$. By the definition of the auxiliary function, $\Theta(\lambda, u) > 0$. Thus, the first argument is non-zero, ensuring $h_t(\lambda, u) \neq 0$. On the base boundary, one has $|\lambda - \lambda_0| = \varepsilon$. If $u=0$, one has $\Theta(\lambda, 0) < 0$, ensuring $h_t(\lambda, 0) \neq 0$. If $u \neq 0$, we analyze the second argument. Since $\mathscr{A}(\lambda)$ is an isomorphism for all $\lambda \in \partial B_\varepsilon(\lambda_0)$ and the boundary circle is compact, the infimum is strictly positive:
\[
c_A := \inf_{\lambda \in \partial B_\varepsilon(\lambda_0)} \|\mathscr{A}(\lambda)^{-1}\|^{-1} > 0.
\]
Setting $\eta = c_A / 2$, the uniform bound established above guarantees that we can choose $\delta > 0$ such that $\|\mathscr{R}(\lambda,u)\| \le \frac{c_A}{2}\|u\|$ for all $\|u\| \le \delta$. Using the reverse triangle inequality on the second argument of $h_t$, we obtain:
\begin{align*}
    \|\mathscr{A}(\lambda)u + t\mathscr{R}(\lambda,u)\| &\ge \|\mathscr{A}(\lambda)u\| - t\|\mathscr{R}(\lambda, u)\| \\
    &\ge c_A\|u\| - \frac{c_A}{2}\|u\| \\
    &= \frac{c_A}{2}\|u\| > 0.
\end{align*}
Since the second component is strictly non-zero, $h_t(\lambda, u) \neq 0$ on the base boundary. The homotopy is therefore $\mathscr{O}$-admissible.
\end{proof}

By the homotopy invariance of the twisted equivariant degree (see \ref{t2}), Lemma~\ref{lem:linearization_homotopy} implies the coincidence $\gdeg(\mathscr{F}_\Theta, \mathscr{O}) = \gdeg(\mathscr{A}_\Theta, \mathscr{O})$ in $A_1^t(G)$. With this in mind, we define the \emph{local bifurcation invariant} at each isolated critical point $(\lambda_0,0) \in \Lambda$ as follows:
\begin{equation}\label{def:local_bifurcation_invariant}
    \omega_G(\lambda_0) := \gdeg(\mathscr{A}_\Theta, \mathscr{O}) \;\in A_1^t(G).
\end{equation}
The following standard Krasnosel'skii-type local bifurcation result (cf. \cite{book-new}, Theorem 6.8) serves as the main engine for our equivariant-degree based local bifurcation analysis. We include a sketch of the proof which highlights the role of the existence property \ref{t4} of the equivariant degree  for the convenience of readers more accustomed to the classical topological degree theory.
\begin{theorem}[Krasnosel'skii-type Local Bifurcation] \rm \label{thm:app_abstract_local_bif}
Suppose that $(\lambda_0,0) \in \Lambda$ is an isolated critical point for \eqref{eq:bif_problem}. If there is an orbit type $(H) \in \Phi_1^t(G)$ for which 
\[
\operatorname{coeff}^H(\omega_{G}(\lambda_0)) \neq 0,
\]
then there exists a branch $\mathscr C$ of non-trivial solutions to \eqref{eq:bif_problem} bifurcating from $(\lambda_0,0)$ with symmetries at least $(H)$, i.e. with $\mathscr C \subset \overline{\mathscr S^H}$.
\end{theorem}
\begin{proof}[Proof Sketch]
By the existence \ref{t4} and homotopy invariance \ref{t2} properties of the twisted equivariant degree, a non-zero coefficient standing next to $(H)$ in the module element \eqref{def:local_bifurcation_invariant} guarantees a zero for the complemented operator $\mathscr{F}_\Theta$ inside the isolating cylinder $\mathscr{O}$ with isotropy at least $(H)$. Since the parameters $\varepsilon$ and $\delta$ defining the isolating cylinder can be chosen to be arbitrarily small, there exist non-trivial solutions accumulating at $(\lambda_0, 0)$. By standard degree-theoretic continuation arguments (analogous to the Rabinowitz global alternative), the non-vanishing of the equivariant degree on these shrinking cylinders forces these zeros to form a connected continuum $\mathscr C$ originating at $(\lambda_0, 0)$. Finally, since the auxiliary function $\Theta$ is strictly positive on $\partial \mathscr{O}$ for any $u \neq 0$, the branch cannot exit the cylinder through the non-trivial boundary, making $\mathscr{C}$ a valid continuum of solutions to the original equation.
\end{proof}
Theorem~\ref{thm:app_abstract_local_bif} reduces the detection of bifurcation to the nonvanishing of the generator coefficients $\operatorname{coeff}^H(\omega_G(\lambda_0))$. In the standard computational framework, evaluating these coefficients requires factoring the local bifurcation invariant through a product of basic degrees and performing the $A(\Gamma)$-module multiplication---a procedure detailed in \cite{AED, book-new, survey}. 

\section{The Complex Spectral Flow} \label{sec:spectral_flow}
For maximal orbit types, the entire algebraic pipeline described above can be bypassed. 
The key is to package the spectral data of the linearization into a virtual $G$-representation in the real representation ring $RO(G)$ from which the coefficients of the local bifurcation invariant at maximal orbit types can be read off directly. We begin by recalling the relevant algebraic framework.
\vs
Let $RO(G)$ denote the real representation ring of $G$. An element $V \in RO(G)$ is a \emph{virtual $G$-representation}, which can be uniquely expressed as a formal finite linear combination of irreducible $G$-representations 
\begin{equation*}
V = \sum_{m=0}^\infty \sum_{j \in \mathcal J} n_{m,j} \mathcal{V}_{m,j}, \quad n_{m,j} \in \bz,
\end{equation*}
where $n_{m,j} = 0$ except for finitely many $G$-isotypic indices. If $n_{m,j} \geq 0$ for all $(m,j)$, then $V$ corresponds to an \emph{actual} (or physical) $G$-representation, which can be viewed geometrically as a direct sum of vector spaces, i.e. $V = \bigoplus_{m=0}^\infty \bigoplus_{j \in \mathcal J} n_{m,j} \mathcal{V}_{m,j}$. However, when coefficients are allowed to be negative, the geometric interpretation requires care. To study the symmetry properties of a virtual representation $V \in RO(G)$, we define its \emph{support} as the actual $G$-representation formed by the direct sum of all its active modes:
\begin{equation*}
\operatorname{supp}(V) := \bigoplus_{n_{m,j} \neq 0} \mathcal{V}_{m,j}.
\end{equation*}
This distinction necessitates a careful differentiation between geometric and algebraic orbit types. For an actual $G$-representation $W$, writing $(H) \in \Phi(G; W)$ implies that there exists a non-zero physical vector $u \in W$ whose isotropy group is exactly $H$. Consequently, the fixed-point space $W^H$ is guaranteed to have a strictly positive dimension. In contrast, for a virtual representation $V \in RO(G)$, we say $(H)$ is an orbit type of $V$ if it is a geometric orbit type of its support, i.e., $(H) \in \Phi(G; \operatorname{supp}(V))$. Keeping in mind that the fixed-point functor distributes linearly, we define the \emph{formal dimension} of the $H$-fixed subspace of $V$ as follows:
\begin{equation*}
\dim V^H := \sum_{m=0}^\infty \sum_{j \in \mathcal J} n_{m,j} \dim \mathcal{V}_{m,j}^H.
\end{equation*}
Notice that the coefficients $n_{m,j}$ can be negative, so
the formal dimension $\dim V^H$ may evaluate to a negative integer or exactly zero, even if $(H)$ is a valid orbit type in the support of $V$.
\subsection{The Local Spectral Flow and the Generalized Kernel}
Since $\mathscr{A}(\lambda) : \mathscr H \to \mathscr H$ is a compact perturbation of the identity, its spectrum consists of isolated eigenvalues of finite algebraic multiplicity. By Schur's Lemma, the $G$-equivariant linearization respects the $G$-isotypic decomposition \eqref{eq:isotypic_decomp_H}, admitting the block structure
\begin{equation*} 
\mathscr{A}(\lambda) = \bigoplus_{m=0}^\infty \bigoplus_{j \in \mathcal{J}} \mathscr{A}_{m,j}(\lambda), \quad \mathscr{A}_{m,j}(\lambda) := \mathscr{A}(\lambda)|_{\mathscr{H}_{m,j}}: \mathscr{H}_{m,j} \to \mathscr{H}_{m,j},
\end{equation*}
so the spectrum of $\mathscr A(\lambda)$ can  be identified as follows
\[
\sigma(\mathscr A(\lambda)) = \bigcup_{m = 0}^\infty \bigcup_{j \in \mathcal J} \sigma(\mathscr A_{m,j}(\lambda)), \quad \sigma(\mathscr A_{m,j}(\lambda)) = \{ \mu_{m,j,k}(\lambda)\}_{k \in \mathcal K_{m,j} \subset \bn}.
\]
At a critical parameter value $\lambda_0$, the generalized kernel $\ker_{gen}\mathscr{A}(\lambda_0)$ is a finite-dimensional $G$-invariant subspace of $\mathscr{H}$ with a $G$-isotypic decomposition of the form
\begin{equation*}
\ker\nolimits_{gen}\mathscr{A}(\lambda_0) \simeq \bigoplus_{m=0}^\infty \bigoplus_{j \in \mathcal J} d_{m,j}(\lambda_0) \mathcal{V}_{m,j},
\end{equation*}
where $d_{m,j}(\lambda_0)$ is the number of indices $k \in \mathcal K_{m,j}$ for which $\mu_{m,j,k}(\lambda_0)=0$.
\vs
An isolated critical point $(\lambda_0,0) \in \Lambda$ is said to be \emph{dynamic} if the restriction $\mathscr A(\lambda_0)|_{\bigoplus_{j \in \mathcal J} \mathscr H_{0,j}}:\bigoplus_{j \in \mathcal J}\mathscr H_{0,j} \to \bigoplus_{j \in \mathcal J}\mathscr H_{0,j}$ is an isomorphism and \emph{stationary}, otherwise.
For any \emph{dynamic} $G$-isotypic component ($m \ge 1$), the non-trivial $S^1$-action induces a natural complex structure. Thus, the restriction $\mathscr{A}(\lambda)|_{\mathscr{H}_{m,j}}$ possesses a finite set $\mathcal K_{m,j}(\lambda_0) \subset \mathcal K_{m,j}$ of parameter-dependent complex eigenvalues $\{\mu_{m,j,k}(\lambda)\}_{k \in \mathcal K_{m,j}(\lambda_0)}$ that vanish at $\lambda_0$. These dynamic eigenvalues cross the origin in the complex plane as the two-dimensional parameter varies, so their topological contribution is captured by their local winding numbers. We define the \emph{complex equivariant spectral flow representation} across the boundary $\partial B_\varepsilon(\lambda_0)$ as the formal sum of the dynamic modes in the representation ring $RO(G)$:
\begin{equation} \label{def:spectral_flow_rep}
\mathbb{W}_{\lambda_0} := \sum_{m=1}^\infty\sum_{j \in \mathcal J} \rho_{m,j}(\lambda_0) \mathcal{V}_{m,j}, \quad \text{where} \quad \rho_{m,j}(\lambda_0) := \sum_{k \in \mathcal K_{m,j}(\lambda_0)} \deg(\mu_{m,j,k}, B_\varepsilon(\lambda_0)).
\end{equation}
Here, the integer $\rho_{m,j}(\lambda_0) \in \bz$ is the \emph{total winding number} (local Brouwer degree) of the critical eigenvalues in the $(m,j)$-th $G$-isotypic component as the parameter traces the boundary circle $\partial B_\varepsilon(\lambda_0)$. These total winding numbers $\rho_{m,j}(\lambda_0)$ can be negative due to topological cancellations, so $\mathbb{W}_{\lambda_0}$ is properly understood as a virtual representation in $RO(G)$. 
\begin{remark}
The restriction to dynamic modes $m \geq 1$ in the definition of the complex spectral flow is structurally necessary. For $m=0$, if the $\Gamma$-representation $\mathcal V_j$
is of real type, the spectrum of $\mathscr A_{m,j}(\lambda_0)$ is symmetric with respect to complex conjugation. Consequently, the winding numbers of any complex conjugate pairs exactly cancel, and strictly real eigenvalues have zero winding number. Thus, the total winding number of a real stationary mode identically vanishes. The topological contribution of the $m=0$ modes is instead captured by the sign of the real determinant, which under an additional assumption factors out as a parity multiplier in the bifurcation invariant.
\end{remark}
In the presence of additional analytical structure---specifically, when the parameter dependence is holomorphic---the possibility of topological cancellation is eliminated, and the spectral flow representation is an actual $G$-representation whose support coincides with the support of the dynamic part of the generalized kernel.
\begin{lemma}\label{lem:holomorphic_kernel}
Assume that the critical eigenvalues of the dynamic $(m \geq 1)$ blocks
$\mathscr A_{m,j}(\lambda)$ depend holomorphically on $\lambda$ in a
neighborhood of the isolated critical point $\lambda_0$. Then the complex
spectral flow representation $\mathbb W_{\lambda_0}$ is an actual
$G$-representation, and its support coincides with the dynamic part of the
generalized kernel:
\[
\operatorname{supp}(\mathbb W_{\lambda_0})
=
\operatorname{supp}\left(
\bigoplus_{m=1}^{\infty}
\bigoplus_{j\in\mathcal J}
d_{m,j}(\lambda_0)\mathcal V_{m,j}
\right).
\]
\end{lemma}
\begin{proof}
For each critical eigenvalue branch $\mu_{m,j,k}(\lambda)$, holomorphy and
isolation of the zero at $\lambda_0$ imply that its winding number around the
origin is the order of vanishing of $\mu_{m,j,k}$ at $\lambda_0$. Hence
$\deg(\mu_{m,j,k},B_\varepsilon(\lambda_0))>0$
whenever $\mu_{m,j,k}(\lambda_0)=0$. Therefore, the total winding number $\rho_{m,j}(\lambda_0)
= \sum_{k\in\mathcal K_{m,j}(\lambda_0)}
\deg(\mu_{m,j,k},B_\varepsilon(\lambda_0))$ is a non-negative integer, and it is strictly positive exactly when the $(m,j)$-block contains a critical eigenvalue at $\lambda_0$.

But the latter condition is precisely the condition that the dynamic
generalized kernel has a nontrivial $(m,j)$-isotypic component, i.e.
$d_{m,j}(\lambda_0)>0$. Thus $\mathbb W_{\lambda_0}$ has only non-negative
coefficients and its active $G$-isotypic components are exactly those constituting the dynamic generalized kernel.
\end{proof}

\subsection{Computation of the Local Bifurcation Invariant}
The spectral flow representation $\mathbb{W}_{\lambda_0}$ encodes the net topological change in the dynamic spectrum.  We now show that, for maximal orbit types, this single object determines the coefficient of the local bifurcation invariant at isolated dynamic critical points---without recourse to basic degrees or Burnside module arithmetic.

\begin{theorem}\label{thm:generalized_guarantee}
Let $(\lambda_0,0) \in \Lambda$ be an isolated dynamic critical point. Suppose that $(H) \in \Phi_1^t(G)$ is a \textbf{maximal} orbit type in the support of the spectral flow representation $\mathbb{W}_{\lambda_0}$. If the formal complex dimension of the $H$-fixed subspace is non-zero, i.e.,
\begin{equation*} 
\dim_\bc \mathbb{W}_{\lambda_0}^H := \sum_{m=1}^\infty \sum_{j \in \mathcal{J}} \rho_{m,j}(\lambda_0) \dim_\bc \mathcal{V}_{m,j}^H \neq 0,
\end{equation*}
then the local bifurcation invariant $\omega_G(\lambda_0)$ admits a strictly non-zero coefficient for $(H)$, given explicitly by the formula:
\begin{equation} \label{eq:universal_formula}
\operatorname{coeff}^H(\omega_{G}(\lambda_0)) = \operatorname{sign}\det\left(\bigoplus_{j \in \mathcal J} \mathscr{A}_{0,j}^H(\lambda_0)\right) \cdot \sum_{m=1}^\infty \sum_{j \in \mathcal{J}} \rho_{m,j}(\lambda_0) \frac{\dim_\bc \mathcal{V}_{m,j}^H}{|W(H)/S^1|} \neq 0,
\end{equation}
where $\bigoplus_{j \in \mathcal J} \mathscr{A}_{0,j}^H(\lambda_0)$ denotes the restriction of the stationary linearization to the $H$-fixed point space $\bigoplus_{j \in \mathcal J}\mathscr H_{0,j}^H$.
\end{theorem}
\begin{proof}
We evaluate the $(H)$-coefficient directly from the recurrence formula \eqref{def:twisted_degree_recurrence} applied to the complemented operator $\mathscr{A}_\Theta$ on the isolating cylinder $\mathscr O$:
\begin{equation} \label{eq:recurrence}
\operatorname{coeff}^H(\omega_G(\lambda_0)) = \frac{\deg_{S^1}(\mathscr{A}_\Theta^H, \mathscr O^H) - \sum_{(L)>(H)} \operatorname{coeff}^L(\omega_G(\lambda_0))\, n(H,L)\, |W(L)/S^1|}{|W(H)/S^1|}.
\end{equation}
The proof proceeds in two stages: we first establish a general identity linking the cumulative $S^1$-degree on any fixed-point space to the formal complex dimension of the spectral flow, and then apply this identity to collapse the recurrence at $(H)$.
\vs
\noi \textbf{Step 1: The Cumulative Degree--Spectral Flow Identity.}
Let $\mathcal M_0 \subset \bn \cup \{0\}$ and $\mathcal J_0 \subset \mathcal J$ be any two subsets of $S^1$- and $\Gamma$-isotypic indices, respectively. For notational convenience, we write $\mathscr A_\Theta[\mathcal M_0, \mathcal J_0](\lambda)u := ( \Theta[\mathcal M_0, \mathcal J_0](\lambda,u),\mathscr A[\mathcal M_0, \mathcal J_0](\lambda)u)$ to indicate the restriction of the complemented operator $\mathscr A_{\Theta}$ to the $G$-isotypic subspace $\mathscr H[\mathcal M_0, \mathcal J_0]:= \overline{\bigoplus_{m \in \mathcal M_0} \bigoplus_{j \in \mathcal J_0}\mathscr H_{m,j}}$, where $\mathscr A[\mathcal M_0, \mathcal J_0](\lambda) := \bigoplus_{m \in \mathcal M_0} \bigoplus_{j \in \mathcal J_0}\mathscr A_{m,j}(\lambda)$ is the restricted linearization and $\Theta[\mathcal M_0, \mathcal J_0](\lambda,u) := \Theta|_{\bc \times \mathscr H[\mathcal M_0, \mathcal J_0]}(\lambda,u)$ is a restricted auxiliary function defined on the restricted isolating cylinder $\mathscr O[\mathcal M_0, \mathcal J_0] := \mathscr O \cap (\bc \times \mathscr H[\mathcal M_0, \mathcal J_0])$.

We claim that, for any twisted orbit type $(L) \in \Phi_1^t(G; \mathscr H)$, one has
\begin{equation}\label{eq:cumulative_identity}
\deg_{S^1}(\mathscr A_\Theta[\bn, \mathcal J]^L, \mathscr O[\bn, \mathcal J]^L) = \dim_\bc \mathbb{W}_{\lambda_0}^L.
\end{equation}
Since $S^1$ lies in the center of $G = S^1 \times \Gamma$, it normalizes every subgroup $L \leq G$. Consequently, the $L$-fixed-point space $\mathscr H[\bn, \mathcal J]^L =\bigoplus_{m=1}^\infty \bigoplus_{j \in \mathcal J} \mathscr H_{m,j}^L$ is an $S^1$-representation, and the restricted operator $\mathscr A_\Theta[\bn, \mathcal J]^L$ is $S^1$-equivariant. 
Since the $S^1$-action on $\mathcal V_{m,j} = \mathcal U_m \otimes \mathcal V_j$ is carried entirely by the factor $\mathcal U_m$ (via the $m$-folded action $e^{im\theta}$), every vector in the $L$-fixed-point space $\mathcal V_{m,j}^L$ transforms under $S^1$ by this same action. Consequently, $\mathcal V_{m,j}^L \simeq (\dim_\bc \mathcal V_{m,j}^L) \cdot \mathcal U_m$ as an $S^1$-representation, and the $m$-th $S^1$-isotypic component of $\mathscr H[\bn, \mathcal J]^L$ admits the $S^1$-isotypic decomposition
\[
\mathscr H[\{m\}, \mathcal J]^L =\bigoplus_{j \in \mathcal J} \mathscr H_{m,j}^L \;\simeq\;  \left( \sum_{j \in \mathcal J} \ell_{m,j}\, \dim_\bc \mathcal V_{m,j}^L\right) \cdot \mathcal U_m.
\]
By the Splitting Lemma (Lemma~\ref{lemm:splitting_lemma}), the total $S^1$-equivariant degree decomposes as the sum of the degrees on the individual isotypic components:
\begin{equation*}
\s1deg(\mathscr A_\Theta[\bn, \mathcal J]^L, \mathscr O[\bn, \mathcal J]^L) = \sum_{m=1}^\infty \s1deg(\mathscr A_\Theta[\{m\}, \mathcal J]^L, \mathscr O[\{m\}, \mathcal J]^L).
\end{equation*}
Applying the determinantal reduction formula (Lemma~\ref{lemm:determinantal_reduction}) directly to the complemented map $\mathscr A_\Theta[\{m\}, \mathcal J]^L$ on each block, we obtain:
\begin{equation*} 
S^1\text{-deg}(\mathscr A_\Theta[\{m\}, \mathcal J]^L, \mathscr O[\{m\}, \mathcal J]^L) = \deg(\det\nolimits_\bc(\mathscr A[\{m\}, \mathcal J]^L), \partial B_\varepsilon(\lambda_0)) \cdot (\bz_m).
\end{equation*}
Here and below, $\det_\bc(\mathscr A[\{m\}, \mathcal J]^L)$ denotes the determinant of the finite-dimensional spectral reduction of $\mathscr A[\{m\}, \mathcal J]^L(\lambda) = \mathscr A(\lambda)|_{\mathscr H[\{m\}, \mathcal J]^L}: \mathscr H[\{m\}, \mathcal J]^L \to \mathscr H[\{m\}, \mathcal J]^L$ to the generalized eigenspaces associated with eigenvalues that vanish at $\lambda_0$; the uniformly invertible complement contributes zero winding. Taking the cumulative $S^1$-degree (i.e., summing the coefficients of all cyclic generators) yields
\begin{align}\label{eq:cumulative_sum}
\deg_{S^1}(\mathscr A_\Theta[\bn, \mathcal J]^L, \mathscr O[\bn, \mathcal J]^L) = \sum_{m=1}^\infty \deg(\det\nolimits_\bc(\mathscr A[\{m\}, \mathcal J]^L), \partial B_\varepsilon(\lambda_0)).
\end{align}
It remains to evaluate each winding number $\deg(\det_\bc(\mathscr A[\{m\}, \mathcal J]^L), \partial B_\varepsilon(\lambda_0))$. On the $m$-th block, the operator decomposes further as $\mathscr A[\{m\}, \mathcal J]^L(\lambda) = \bigoplus_j \mathscr A_{m,j}^L(\lambda)$, so the complex determinant factors accordingly to
\[
\det\nolimits_\bc(\mathscr A[\{m\}, \mathcal J]^L(\lambda)) = \prod_{j \in \mathcal J} \det\nolimits_\bc(\mathscr A_{m,j}^L(\lambda)).
\]
By Schur's Lemma, the $G$-equivariant operator $\mathscr A_{m,j}(\lambda)$ acts on the $G$-isotypic component $\mathscr H_{m,j} \simeq \ell_{m,j}\, \mathcal V_{m,j}$ as $\id_{\mathcal V_{m,j}} \otimes B_{m,j}(\lambda)$, where $B_{m,j}(\lambda)$ is an $\ell_{m,j} \times \ell_{m,j}$ complex matrix with eigenvalues $\{\mu_{m,j,k}(\lambda)\}_k$. On the $L$-fixed subspace $\mathscr H_{m,j}^L \simeq \ell_{m,j}\, \mathcal V_{m,j}^L$, the restricted operator $\mathscr A_{m,j}^L(\lambda)$ acts as $\id_{\mathcal V_{m,j}^L} \otimes B_{m,j}(\lambda)$, yielding the determinant
\[
\det\nolimits_\bc(\mathscr A_{m,j}^L(\lambda)) = \left(\prod_{k} \mu_{m,j,k}(\lambda)\right)^{\dim_\bc \mathcal V_{m,j}^L}.
\]
Since the winding number is additive under products and multiplicative under powers, one obtains
\[
\deg(\det\nolimits_\bc(\mathscr A_{m,j}^L), \partial B_\varepsilon(\lambda_0)) = \dim_\bc \mathcal V_{m,j}^L \cdot \sum_{k \in \mathcal K_{m,j}(\lambda_0)} \deg(\mu_{m,j,k}, B_\varepsilon(\lambda_0)) = \dim_\bc \mathcal V_{m,j}^L \cdot \rho_{m,j}(\lambda_0).
\]
Summing over $j$ yields $\deg(\det_\bc(\mathscr A[\{m\}, \mathcal J]^L), \partial B_\varepsilon(\lambda_0)) = \sum_{j \in \mathcal J} \rho_{m,j}(\lambda_0)\, \dim_\bc \mathcal V_{m,j}^L$, and substituting into \eqref{eq:cumulative_sum} recovers
\[
\deg_{S^1}(\mathscr A_\Theta[\bn, \mathcal J]^L, \mathscr O[\bn, \mathcal J]^L) = \sum_{m=1}^\infty \sum_{j \in \mathcal J} \rho_{m,j}(\lambda_0)\, \dim_\bc \mathcal V_{m,j}^L = \dim_\bc \mathbb W_{\lambda_0}^L,
\]
establishing the identity \eqref{eq:cumulative_identity}.
\vs
\noi \textbf{Step 2: Collapse of the Recurrence at the Maximal Orbit Type.}
By the product property of the cumulative $S^1$-degree (Lemma~\ref{lemm:s1_degree_product_property}), which allows a direct sum with a fixed-point space on which $S^1$ acts trivially to contribute only a local Brouwer-degree sign, the cumulative degree on the $L$-fixed-point space factors as
\begin{equation}\label{eq:product_factorization}
\deg_{S^1}(\mathscr{A}_\Theta^L, \mathscr O^L) = \operatorname{sign}\det\left(\bigoplus_{j \in \mathcal{J}} \mathscr{A}_{0,j}^L(\lambda_0)\right) \cdot \deg_{S^1}(\mathscr A_\Theta[\bn, \mathcal J]^L, \mathscr O[\bn, \mathcal J]^L).
\end{equation}
Since $(H)$ is maximal in the support of the spectral flow representation $\mathbb{W}_{\lambda_0}$, any strictly larger orbit type $(L) > (H)$ satisfies $\dim_\bc \mathbb{W}_{\lambda_0}^L = 0$. 
Indeed, for any $(L) > (H)$ in $\Phi_1^t(G)$, the maximality of $(H)$ in the support implies that $\operatorname{supp}(\mathbb{W}_{\lambda_0})^L 
\subseteq \operatorname{supp}(\mathbb{W}_{\lambda_0})^G = \{0\}$ (the latter since all active modes satisfy $m \geq 1$), so $\dim_\bc \mathcal{V}_{m,j}^L = 0$ for every $m\geq 1$ and the formal dimension vanishes termwise. By the identity \eqref{eq:cumulative_identity}, it follows that
\[
\deg_{S^1}(\mathscr A_\Theta[\bn, \mathcal J]^L, \mathscr O[\bn, \mathcal J]^L) = 0 \quad \implies \quad \deg_{S^1}(\mathscr{A}_\Theta^L, \mathscr O^L) = 0 \quad \text{ for all } \quad  (L) > (H).
\]
Proceeding inductively from the top of the isotropy lattice via the recurrence formula \eqref{def:twisted_degree_recurrence},  one has $\operatorname{coeff}^L(\omega_G(\lambda_0)) = 0$ for all $(L) > (H)$, such that the summation in \eqref{eq:recurrence} vanishes, and the coefficient formula collapses to
\begin{equation*} 
\operatorname{coeff}^H(\omega_G(\lambda_0)) = \frac{\deg_{S^1}(\mathscr{A}_\Theta^H, \mathscr O^H)}{|W(H)/S^1|}.
\end{equation*}
Applying the factorization \eqref{eq:product_factorization} and the identity \eqref{eq:cumulative_identity} with $L = H$, one obtains
\begin{equation*} 
\operatorname{coeff}^H(\omega_G(\lambda_0)) = \operatorname{sign}\det\left(\bigoplus_{j \in \mathcal J} \mathscr{A}_{0,j}^H(\lambda_0)\right) \cdot \frac{\dim_\bc \mathbb{W}_{\lambda_0}^H}{|W(H)/S^1|} \neq 0,
\end{equation*}
where the final inequality holds since $\dim_\bc \mathbb W_{\lambda_0}^H \neq 0$ by hypothesis and $\operatorname{sign}\det(\cdot) = \pm 1$ since the critical point is dynamic. Expanding the definition of $\dim_\bc \mathbb{W}_{\lambda_0}^H$ recovers the explicit computational formula \eqref{eq:universal_formula}.
\end{proof}

\begin{corollary}\label{cor:holomorphic_universal_guarantee}
Suppose that the critical dynamic eigenvalues of $\mathscr{A}(\lambda)$ are
holomorphic in a neighborhood of an isolated dynamic critical point
$(\lambda_0,0)\in\Lambda$. If
$(H)\in \Phi_1^t(G;\ker_{gen}\mathscr A(\lambda_0))$ is a \textbf{maximal}
orbit type in the generalized kernel, then the coefficient formula
\eqref{eq:universal_formula} for
$\operatorname{coeff}^H(\omega_G(\lambda_0))$ is unconditionally non-zero.
\end{corollary}
\begin{proof}
Since $(\lambda_0,0)$ is dynamic, the generalized kernel has no stationary
part. By Lemma~\ref{lem:holomorphic_kernel}, the spectral flow representation
$\mathbb W_{\lambda_0}$ is an actual $G$-representation that shares the same support as the dynamic generalized kernel. Hence $(H)$ is also a maximal orbit type in the support of $\mathbb W_{\lambda_0}$.

It remains only to verify that the formal fixed-point dimension is non-zero.
Since $(H)$ occurs in the support of $\mathbb W_{\lambda_0}$, there exists a
non-zero vector in $\operatorname{supp}(\mathbb W_{\lambda_0})$ fixed by $H$.
Thus, for at least one active index $(m,j)$, one has both $\rho_{m,j}(\lambda_0)>0$ and $\dim_\bc \mathcal V_{m,j}^H>0$.
Moreover, Lemma~\ref{lem:holomorphic_kernel} gives
$\rho_{m,j}(\lambda_0)\geq 0$ for every dynamic block. Therefore every term in
the formal dimension
\[
\dim_\bc \mathbb W_{\lambda_0}^H
=
\sum_{m=1}^{\infty}\sum_{j\in\mathcal J}
\rho_{m,j}(\lambda_0)\dim_\bc\mathcal V_{m,j}^H
\]
is non-negative, and at least one term is strictly positive. Hence
$\dim_\bc \mathbb W_{\lambda_0}^H>0$. Theorem~\ref{thm:generalized_guarantee}
therefore applies and gives $\operatorname{coeff}^H(\omega_G(\lambda_0))\neq 0$.
\end{proof}
Having established the local theory---both in general (Theorem~\ref{thm:generalized_guarantee}) and in the holomorphic setting (Corollary~\ref{cor:holomorphic_universal_guarantee})---we now turn to the question of global continuation.
\subsection{Resolution of the Rabinowitz Alternative}
\label{sec:unbounded_branches}
The preceding results determine the coefficient of the local bifurcation invariant at a single isolated critical point. We now turn to global continuation, asking: when is a bifurcating branch forced to escape any open, bounded $G$-invariant domain $\mathcal U \subset \bc \times \mathscr H$ satisfying $\partial \mathcal U \cap \Lambda = \emptyset$? The argument proceeds in three layers, each building on the last. 
\begin{enumerate}
    \item \textbf{Topological escape criterion} (Theorem~\ref{thm:rab_alt} and Corollary~\ref{cor:macro_escape}). The equivariant Rabinowitz alternative reduces macroscopic escape to the nonvanishing of an aggregate local bifurcation invariant $\sum_{(\lambda_i,0) \in \mathcal U \cap \Lambda} \omega_G(\lambda_i)$.
    \item \textbf{Spectral translation} (Theorem~\ref{thm:macroscopic_escape} and Corollary~\ref{cor:holomorphic_escape}). The local spectral flows over $\mathcal U$ are assembled into signed aggregate spectral flows by weighting each critical point with the stationary determinant sign on the relevant fixed-point space. If $(H)$ is maximal in the aggregate support, then all higher signed charges vanish automatically; macroscopic escape is reduced to the nonvanishing of the $H$-signed formal fixed-point dimension.
     \item \textbf{Spatio-temporal reduction} (Corollary~\ref{cor:spatiotemporal_guarantee}, Theorem~\ref{thm:spatio_temporal_macroscopic_escape}, and Corollary~\ref{cor:spatio_temporal_holomorphic_escape}). Restricting to an anti-periodic fixed-point subspace eliminates the stationary modes, so that the escape criterion reduces to a single formal-dimension condition on the aggregate spectral flow. 
\end{enumerate}
The intuition behind the first layer is as follows. If a bifurcating branch $\mathscr C$ remains trapped inside a bounded domain $\mathcal U$, it can only terminate on other critical points inside $\mathcal U$, and the equivariant degree contributions at these endpoints must cancel in aggregate. Contrapositively, a nonzero aggregate invariant
forces at least one branch to reach the boundary. The following result makes this precise; it is the equivariant analogue of the classical  Rabinowitz alternative (see~\cite{Rabinowitz}) and is standard in the 
context of the equivariant degree theory (see~\cite{AED, book-new, survey}). 
\begin{theorem} \rm \label{thm:rab_alt}{\bf (Rabinowitz' Alternative)}
Let $\mathcal U \subset \bc \times \mathscr H$ be any open bounded $G$-invariant set with $\partial \mathcal U \cap \Lambda = \emptyset$. If $\mathscr C \subset \overline{\mathscr S}$ is a branch of nontrivial solutions to \eqref{eq:bif_problem} bifurcating from a critical point $(\lambda_i,0) \in \mathcal U \cap \Lambda$ satisfying $\omega_G(\lambda_i) \neq 0$, then one has the following alternative:
\begin{enumerate}[label=$(\alph*)$]
\item\label{alt_a}  either $\mathscr C \cap \partial \mathcal U \neq \emptyset$;
    \item\label{alt_b} or there exists a finite set $\mathscr C \cap \Lambda = \{ (\lambda_1,0), \ldots, (\lambda_{N},0) \}$
    satisfying the following relation
    \begin{align*}
\sum\limits_{i=1}^{N}\omega_G(\lambda_i) = 0.
    \end{align*}
\end{enumerate} 
\end{theorem}
We impose a standing critical-set hypothesis, including a non-emptiness assumption to ensure that the problem has critical parameters from which bifurcation may occur and a discreteness assumption to guarantee that each critical point is isolated, so that the local bifurcation invariant $(\omega_G(\lambda_0)$ is well defined after choosing a sufficiently small
isolating disk around $\lambda_0$:
\begin{enumerate}[label=($B_0$)] 
    \item\label{b0} $\Lambda \neq \emptyset$ and $\Lambda$ is discrete.
\end{enumerate}
\begin{corollary}\label{cor:macro_escape}
Let $\mathcal U\subset \bc\times\mathscr H$ be open, bounded and $G$-invariant, with $\mathcal U\cap\Lambda=\{(\lambda_1,0),\ldots, (\lambda_N,0)\}$ and $\partial\mathcal U\cap\Lambda=\emptyset$. If one has
\[
\sum_{i=1}^N \omega_G(\lambda_i)\neq 0,
\]
then there exists a maximal connected component
$\mathscr C\subset\overline{\mathscr S}$ satisfying $\mathscr C\cap\partial\mathcal U\neq\emptyset$.
\end{corollary}
\begin{proof}
Suppose, for the sake of contradiction, that no maximal connected component of
$\overline{\mathscr S}$ containing a critical point in $\mathcal U\cap\Lambda$ intersects $\partial\mathcal U$. We partition the critical points $\{(\lambda_1,0),\ldots,(\lambda_N,0)\}$ according to their membership in the maximal connected components of $\overline{\mathscr{S}}$. In particular, 
let $\mathscr C_1,\ldots,\mathscr C_p$ be the distinct maximal connected components of $\overline{\mathscr{S}}$ that contain at least one critical point from $\mathcal{U} \cap \Lambda$, and define
\[
\Lambda_k := \{(\lambda_i,0) \in \mathcal{U} \cap \Lambda : (\lambda_i,0) \in \mathscr C_k\}, \quad k = 1,\ldots,p.
\]
Since the maximal connected components of a topological space are pairwise disjoint, the sets $\Lambda_1,\ldots,\Lambda_p$ are also pairwise disjoint. Let $\Lambda_0 := \{(\lambda_i,0)\}_{i=1}^N \setminus \bigcup_{k=1}^p \Lambda_k$ denote the remaining critical points, which are not limit points of any continuum of non-trivial solutions. We claim that the total sum of local bifurcation invariants vanishes, contradicting the hypothesis $\sum_{i=1}^N \omega_G(\lambda_i)\neq 0$.
\vs
\noi\textit{Critical points in $\Lambda_0$:} If any $(\lambda_i,0) \in \Lambda_0$ satisfied $\omega_{G}(\lambda_i) \neq 0$, then Theorem~\ref{thm:app_abstract_local_bif} would guarantee a branch of non-trivial solutions emerging from $(\lambda_i,0)$, placing $(\lambda_i,0)$ in $\Lambda_k$ for some $k \in \{1,\ldots,p\}$, a contradiction. Therefore $\omega_G(\lambda_i) = 0$ for every $(\lambda_i,0) \in \Lambda_0$.
\vs
\noi\textit{Critical points in $\Lambda_k$ $(k \geq 1)$:}
Fix any $k\in\{1,\ldots,p\}$. If $\omega_G(\lambda)=0$ for every $(\lambda,0)\in\Lambda_k$ then clearly $\sum_{(\lambda,0)\in\Lambda_k}\omega_G(\lambda)=0$. Otherwise, there exists at least one critical point $(\lambda_*,0)\in\Lambda_k$ with $\omega_G(\lambda_*)\neq 0$. The maximal connected component $\mathscr C_k$ of $\overline{\mathscr S}$ is then a branch of non-trivial solutions bifurcating from $(\lambda_*,0)$. By our
contradiction hypothesis, $\mathscr C_k\cap\partial\mathcal U=\emptyset$.
Since $\mathscr C_k$ emerges from the trivial branch in $\mathcal U$ and does not intersect $\partial\mathcal U$, connectivity implies $\mathscr C_k\subset\mathcal U$. Therefore the Rabinowitz alternative (Theorem~\ref{thm:rab_alt}) applies to $\mathscr C_k$. Since alternative~\ref{alt_a} is impossible,
alternative~\ref{alt_b} yields
\[
\sum_{(\lambda,0)\in\Lambda_k}\omega_G(\lambda)=0.
\]
Thus, in either case, one has $\sum_{(\lambda,0)\in\Lambda_k}\omega_G(\lambda)=0$ for all $k=1,\ldots,p$. Summing over all critical points in $\mathcal{U}$, one has
\[
\sum_{i=1}^N \omega_G(\lambda_i) = \underbrace{\sum_{(\lambda,0) \in \Lambda_0} \omega_G(\lambda)}_{= 0} + \sum_{k=1}^p \underbrace{\sum_{(\lambda,0) \in \Lambda_k} \omega_G(\lambda)}_{= 0} = 0,
\]
which directly contradicts our hypothesis. Therefore, at least one branch $\mathscr C$ must intersect $\partial\mathcal{U}$.
\end{proof}

\subsubsection{The Signed Aggregate Spectral Flow}
Theorem~\ref{thm:rab_alt} and Corollary~\ref{cor:macro_escape} are entirely topological: they guarantee escape whenever the aggregate bifurcation invariant is nonzero, but do not indicate how to verify this condition from spectral data. We now bridge this gap by introducing a signed aggregate spectral flow whose formal fixed-point dimensions recover the coefficients of the aggregate invariant at maximal orbit types.
\vs
Let $\mathcal U\subset\bc\times\mathscr H$ be any open, bounded and $G$-invariant set satisfying $\partial\mathcal U\cap\Lambda=\emptyset$ and assume that $\mathcal U\cap\Lambda$ is comprised of a finite number of dynamic critical points, such that the stationary block $\bigoplus_{j\in\mathcal J}\mathscr A_{0,j}(\lambda_i) : \bigoplus_{j\in\mathcal J} \mathscr H_{0,j} \to \bigoplus_{j\in\mathcal J}\mathscr H_{0,j}$ is invertible for each $(\lambda_i,0) \in \mathcal U\cap\Lambda$. We first record the \emph{unsigned aggregate support}
\[
\mathbb S_{\mathcal U} := \bigoplus_{(\lambda_i,0)\in\mathcal U\cap\Lambda} \operatorname{supp}(\mathbb W_{\lambda_i}),
\]
then,  for any twisted orbit type $(K)\in\Phi_1^t(G)$, we define the $K$-signed aggregate spectral flow over $\mathcal U$ by
\[
\mathbb W_{\mathcal U}(K)
:= \sum_{(\lambda_i,0)\in\mathcal U\cap\Lambda}
\operatorname{sign}\det_\br
\left( \bigoplus_{q\in\mathcal J}
\mathscr A_{0,q}^K(\lambda_i)
\right) \mathbb W_{\lambda_i}.
\]
Expanding the definition of the local spectral flows $\mathbb W_{\lambda_i}$ gives the formal fixed-point dimensions
\begin{align*}
\dim_\bc \left(\mathbb W_{\mathcal U}(K) \right)^K
&=
\sum_{(\lambda_i,0)\in\mathcal U\cap\Lambda}
\operatorname{sign}\det_\br \left(\bigoplus_{q\in\mathcal J}
\mathscr A_{0,q}^K(\lambda_i) \right)
\dim_\bc\mathbb W_{\lambda_i}^K \\
&=
\sum_{m\geq 1}\sum_{j\in\mathcal J}
\left( \sum_{(\lambda_i,0)\in\mathcal U\cap\Lambda} \operatorname{sign}\det_\br
\left( \bigoplus_{q\in\mathcal J}
\mathscr A_{0,q}^K(\lambda_i) \right)
\rho_{m,j}(\lambda_i) \right)
\dim_\bc\mathcal V_{m,j}^K .
\end{align*}

\begin{theorem}\label{thm:macroscopic_escape}
Let $\mathcal U\subset\bc\times\mathscr H$ be open, bounded and
$G$-invariant set with $\partial\mathcal U\cap\Lambda=\emptyset$ and
$\mathcal U\cap\Lambda=\{(\lambda_1,0),\ldots,(\lambda_N,0)\}$ consisting only of dynamic critical points. If $(H)\in\Phi_1^t(G)$ is a maximal orbit type in the unsigned aggregate support $\mathbb S_{\mathcal U}$ and
\[
\dim_\bc
\left(
\mathbb W_{\mathcal U}(H)
\right)^H
\neq 0,
\]
then there exists a maximal connected component $\mathscr C\subset\overline{\mathscr S}$ such that $\mathscr C\cap\partial\mathcal U\neq\emptyset$.
\end{theorem}
\begin{proof}
For each $(K)\in\Phi_1^t(G)$, put
\[
C_{\mathcal U}(K)
:=
\sum_{(\lambda_i,0)\in\mathcal U\cap\Lambda}
\operatorname{coeff}^K(\omega_G(\lambda_i)).
\]
Summing the local recurrence formula for the twisted equivariant degree gives
\[
C_{\mathcal U}(K)
=
\frac{
\dim_\bc \left(\mathbb W_{\mathcal U}(K) \right)^K - \sum_{(L)>(K)}
n(K,L)|W(L)/S^1|\,C_{\mathcal U}(L)}
{|W(K)/S^1|}.
\]
Since $(H)$ is maximal in the support of $\mathbb S_{\mathcal U}$, one has $\mathbb S_{\mathcal U}^L=\{0\}$ for every $(L)>(H)$. Every signed aggregate spectral flow $\mathbb W_{\mathcal U}(L)$ is supported inside $\mathbb S_{\mathcal U}$, so
\[
\dim_\bc \left( \mathbb W_{\mathcal U}(L)
\right)^L =0
\qquad \text{for all } (L)>(H).
\]
Downward induction over the orbit-type partial order therefore gives $C_{\mathcal U}(L)=0$ for all $(L)>(H)$. Evaluating the recurrence at $(H)$ yields
\[
C_{\mathcal U}(H)
= \frac{ \dim_\bc \left(\mathbb W_{\mathcal U}(H) \right)^H} {|W(H)/S^1|}
= \operatorname{coeff}^H \left( \sum_i\omega_G(\lambda_i) \right) \neq 0,
\]
where the final inequality follows from the hypothesis $\dim_\bc(\mathbb W_{\mathcal U}(H))^H\neq0$. The conclusion follows from Corollary~\ref{cor:macro_escape}.
\end{proof}
When the eigenvalue dependence is holomorphic, the aggregate flow representation simplifies further. To state this precisely, we introduce the \emph{aggregate dynamic generalized kernel}:
\[
\mathbb K_{\mathcal U}
:=
\bigoplus_{(\lambda_i,0)\in\mathcal U\cap\Lambda} \bigoplus_{m=1}^{\infty} \bigoplus_{j\in\mathcal J} d_{m,j}(\lambda_i)\mathcal V_{m,j}.
\]
\begin{corollary}\label{cor:holomorphic_escape}
Let $\mathcal U\subset\bc\times\mathscr H$ be open, bounded and $G$-invariant, with $\partial\mathcal U\cap\Lambda=\emptyset$, and suppose $\mathcal U\cap\Lambda$ consists of finitely many dynamic critical points. Assume that the critical eigenvalues of the dynamic blocks $\mathscr A_{m,j}(\lambda)$ depend holomorphically on $\lambda$ near each critical point in $\mathcal U\cap\Lambda$. 
If a twisted orbit type $(H)\in\Phi_1^t(G)$ is maximal in the support of $\mathbb K_{\mathcal U}$, and if the stationary sign
\[
\operatorname{sign} \det_\br
\left(\bigoplus_{j\in\mathcal J} \mathscr A_{0,j}^H(\lambda_i) \right)
\]
is constant over all $(\lambda_i,0)\in\mathcal U\cap\Lambda$ satisfying $\dim_\bc\mathbb W_{\lambda_i}^H>0$, then there exists a maximal connected component $\mathscr C\subset\overline{\mathscr S}$ such that
$\mathscr C\cap\partial\mathcal U\neq\emptyset$.
\end{corollary}
\begin{proof}
By Lemma~\ref{lem:holomorphic_kernel}, each local spectral flow $\mathbb W_{\lambda_i}$ is an actual $G$-representation, and its support coincides with the support of the dynamic generalized kernel of $\mathscr A(\lambda_i)$. Hence the aggregate dynamic support $\mathbb S_{\mathcal U}$ agrees with the support of $\mathbb K_{\mathcal U}$. In particula, the maximality of $(H)$ in the support of $\mathbb K_{\mathcal U}$ implies its maximality in the support of $\mathbb S_{\mathcal U}$. Thus Theorem~\ref{thm:macroscopic_escape} reduces the proof to showing that
\[
\dim_\bc
\left(
\mathbb W_{\mathcal U}(H)
\right)^H
\neq 0.
\]
By holomorphy, all winding numbers $\rho_{m,j}(\lambda_i)$ are non-negative, and $\rho_{m,j}(\lambda_i)>0$ precisely on the active dynamic kernel components. Therefore every quantity $\dim_\bc\mathbb W_{\lambda_i}^H$ is non-negative, and at least one is strictly positive because $(H)$ occurs in the aggregate kernel.

By the assumed constancy of the stationary $H$-sign on the nonzero $H$-fixed contributions, the signed formal dimension is either the positive sum $\sum_{(\lambda_i,0)\in\mathcal U\cap\Lambda} \dim_\bc\mathbb W_{\lambda_i}^H$ or the negative of this quantity. In either case it is nonzero:
\[
\dim_\bc \left( \mathbb W_{\mathcal U}(H) \right)^H
=
\sum_{(\lambda_i,0)\in\mathcal U\cap\Lambda}
\operatorname{sign}\det_\br \left(
\bigoplus_{q\in\mathcal J} \mathscr A_{0,q}^H(\lambda_i) \right) \dim_\bc\mathbb W_{\lambda_i}^H \neq 0.
\]
The conclusion follows from Theorem~\ref{thm:macroscopic_escape}.
\end{proof}
\subsubsection{Spatio-Temporal Fixed-Point Reduction}
The abstract bifurcation framework of Sections~\ref{sec:bif_framework}--\ref{sec:spectral_flow} is formulated for a general product group $G = S^1 \times \Gamma$, without specifying the origin of the $S^1$ factor. In practice, the $S^1$-symmetry frequently arises from the time-shift invariance of periodic orbits: if $x(t)$ is a $p$-periodic solution of an autonomous system, then so is $x(t + \tau)$ for any phase shift $\tau \in S^1 \simeq \br / p\bz$. In this setting, any constant (equilibrium) solution $x(t) \equiv x_0$ is trivially $p$-periodic for every $p > 0$ and is fixed by the entire 
$S^1$-action. Such equilibria therefore appear as spurious stationary bifurcating branches, obscuring the genuinely time-dependent dynamics one seeks to detect. Moreover, their presence introduces the stationary determinant sign into the coefficient formula~\eqref{eq:universal_formula}, which can vary across critical points and obstruct cancellation-free global guarantees.
\vs
Both issues are resolved simultaneously by restricting to a fixed-point subspace that does not admit the $m = 0$ Fourier modes in its $G$-isotypic decomposition. When the symmetry group admits a spatial involution, a spatio-temporal anti-periodic symmetry provides exactly such a filter: it retains only the odd-harmonic modes ($m = 1, 3, 5, \ldots$), ensuring that every nontrivial solution in the 
reduced space is genuinely time-dependent, and that the escape criterion reduces to a pure dimension condition on the spectral flow---with no sign verification required.
\vs
\noi\textbf{The anti-periodic subgroup.}
Assume the full symmetry group admits a spatial involution, allowing us to write $G = S^1 \times \Gamma \times \bz_2$, and define the spatio-temporal subgroup
\[
\bm{H} := \langle (-1, e_\Gamma, -1) \rangle \leq G,
\]
generated by the simultaneous half-period time shift and spatial sign reversal. Consider the $\bm H$-fixed point subspace $\mathscr H^{\bm H} := \{ u \in \mathscr H: h u = u \; \forall_{h \in \bm H} \}$ and the restrictions 
$\mathscr F^{\bm H} := \mathscr F|_{\br \times \br \times \mathscr H^{\bm H}}$ and $\mathscr A^{\bm H} := \mathscr A|_{\br \times \br \times 
\mathscr H^{\bm H}}$. Any solution $(\alpha,\beta,u) \in \br \times \br \times \mathscr H^{\bm H}$ to the $\bm H$-fixed operator equation
\begin{align} \label{eq:H_fixed_bif_problem}
    \mathscr F^{\bm H}(\alpha,\beta,u) = 0,
\end{align}
is also a solution to~\eqref{eq:bif_problem}.
\vs
\noi\textbf{Isotypic decomposition of $\mathscr H^{\bm H}$ and the exclusion of stationary solutions.}
Since $\bm{H}$ is central in $G$ (the $\bz_2$ factor is direct), its normalizer is $N(\bm{H}) = G$ and the residual symmetry group is $\bm{G} := W(\bm{H}) = G/\bm{H} \simeq S^1 \times \Gamma$. If $\mathcal V_j^-$ denotes the irreducible $\Gamma \times \bz_2$-representation equipped with the antipodal 
$\bz_2$-action, the generator $(-1, e_\Gamma, -1) \in \bm{H}$ acts on $\mathcal U_m \otimes \mathcal V_j^-$ by the phase $(-1)^m(-1)$. This equals $+1$ precisely when $m$ is odd, so the 
$\bm H$-fixed point space admits the $\bm G$-isotypic decomposition
\[
\mathscr H^{\bm H} = \overline{\bigoplus_{m \in 2 \bn -1} \bigoplus_{j \in \mathcal J} \mathscr H_{m,j}}, \quad \mathscr H_{m,j} \simeq \ell_{m,j} \mathcal V_{m,j}.
\]
Notice that every connected component of nontrivial solutions to the $\bm{H}$-fixed problem consists entirely of genuinely time-dependent orbits. Indeed, an element $u \in \mathscr H$ is $S^1$-invariant if and only if $u \in \bigoplus_j \mathscr H_{0,j}$, so any nontrivial solution in $\mathscr H^{\bm H}$ is not fixed by $S^1$. Moreover, any trivial solution $(\lambda,0)$ at which $\mathscr A^{\bm H}(\lambda)$ is singular is automatically a dynamic critical point of the original equation~\eqref{eq:bif_problem}.
\vs
To ensure the local bifurcation invariants are well-defined in this reduced setting, we replace assumption \ref{b0} with the following:
\begin{enumerate} [label=($B_1$)] 
    \item\label{b1} the $\bm H$-fixed critical set $\Lambda^{\bm H} := \Lambda \cap \bc \times \mathscr H^{\bm H}$ is nonempty and discrete.
\end{enumerate}
Given an isolating cylinder $\mathscr O^{\bm H}$, a linear complemented operator $\mathscr A_\Theta^{\bm H}$, and an auxiliary function $\Theta$ (as defined as in Section~\ref{sec:bif_framework}, now adapted to the $\bm H$-fixed setting), we define the $\bm H$-fixed local bifurcation invariant at an isolated critical point $(\lambda_0,0) \in \Lambda^{\bm H}$ by
\[
\omega_{\bm G}(\lambda_0) := \bm G\text{-deg}
(\mathscr A_{\Theta}^{\bm H}, \mathscr O^{\bm H}) \in \Phi_1^t(\bm G).
\]
Since the generalized kernel of $\mathscr{A}^{\bm{H}}(\lambda_0)$ lives entirely in $\mathscr{H}^{\bm{H}}$, the \emph{$\bm H$-fixed virtual spectral flow representation} in $RO(\bm{G})$ only admits active modes with odd $S^1$-isotypic indices:
\begin{equation*} 
\mathbb{W}_{\lambda_0}^{\bm{H}} := \sum_{m \in 2\bn-1} \sum_{j \in \mathcal{J}} \rho_{m,j}(\lambda_0) \mathcal{V}_{m,j}.
\end{equation*}
The arguments used in Theorem \ref{thm:generalized_guarantee} apply directly to this reduced space. Since the temporal fixed-point reduction eliminates all stationary modes ($m=0$), the stationary linearization restricts to the origin $\{0\}$ in $\mathscr{H}^{\bm{H}}$. Consequently, the stationary determinant evaluates to $+1$, and the coefficient formula~\eqref{eq:universal_formula} collapses to the formal dimension divided by $|W(H)/S^1|$, with no multiplicative sign factor.

\begin{corollary}\label{cor:spatiotemporal_guarantee}
Under the assumption \ref{b1}, let $(\lambda_0,0) \in\Lambda^{\bm H}$. Suppose $(H) \in \Phi_1^t(\bm{G})$ is a \textbf{maximal} orbit type in the support of the $\bm{H}$-fixed virtual spectral flow representation $\mathbb{W}_{\lambda_0}^{\bm{H}}$. If the formal complex dimension of the $H$-fixed subspace is non-zero, i.e.,
\begin{equation*}
\dim_\bc (\mathbb{W}_{\lambda_0}^{\bm{H}})^H := \sum_{m \in 2\bn-1} \sum_{j \in \mathcal{J}} \rho_{m,j}(\lambda_0) \dim_\bc \mathcal{V}_{m,j}^H \neq 0,
\end{equation*}
then the $\bm{H}$-fixed local bifurcation invariant admits a strictly non-zero coefficient for $(H)$, given explicitly by the formula:
\begin{equation*} 
\operatorname{coeff}^H(\omega_{\bm{G}}(\lambda_0)) = \sum_{m \in 2\bn-1} \sum_{j \in \mathcal{J}} \rho_{m,j}(\lambda_0) \frac{\dim_\bc \mathcal{V}_{m,j}^H}{|W(H)/S^1|} \neq 0.
\end{equation*}
\end{corollary}

\subsubsection{The Aggregate Spectral Flow}
Corollary~\ref{cor:spatiotemporal_guarantee} is a local guarantee at a single, isolated $\bm H$-fixed critical point. We now aggregate over all critical points in a bounded parameter domain to obtain a global escape criterion in the reduced setting.
\vs
Let $\mathcal{U} \subset \bc \times \mathscr{H}^{\bm{H}}$ be an open, bounded, $\bm{G}$-invariant set such that its boundary contains no critical points (i.e., $\partial \mathcal{U} \cap \Lambda^{\bm{H}} = \emptyset$). 
To capture the total topological charge within this boundary, we define the \emph{aggregate virtual spectral flow} over $\mathcal{U}$ as the formal sum of the localized $\bm H$-fixed virtual flows in $RO(\bm{G})$:
\begin{equation*}
\mathbb{W}_{\mathcal{U}} := \sum_{(\lambda_i,0) \in \Lambda^{\bm H} \cap \,\mathcal U} \mathbb{W}_{\lambda_i}^{\bm H} = \sum_{m \in 2 \bn - 1} \sum_{j \in \mathcal{J}} \left( \sum_{(\lambda_i,0) \in \Lambda^{\bm H} \cap \,\mathcal U} \rho_{m,j}(\lambda_i) \right) \mathcal{V}_{m,j}.
\end{equation*}
Under the assumption \ref{b1}, $\mathbb{W}_{\mathcal{U}}$ is a well-defined virtual $\bm G$-representation. Since the $\bm H$-fixed operator equation~\eqref{eq:H_fixed_bif_problem} is defined with respect to a completely continuous $\bm G$-equivariant field on $\mathscr H^{\bm H}$, the Rabinowitz alternative (Theorem~\ref{thm:rab_alt}) and Corollary~\ref{cor:macro_escape} apply \emph{mutatis mutandis} to the restricted problem, with the symmetry group $\bm G$, the phase space $\mathscr H^{\bm H}$, and the bifurcation invariant $\omega_{\bm G}$, in place of $G$, $\mathscr H$ and $\omega_G$.
\begin{theorem}\label{thm:spatio_temporal_macroscopic_escape}
Suppose that $\Lambda^{\bm H}$ is discrete and let $\mathcal{U} \subset \bc \times \mathscr{H}^{\bm{H}}$ be an open, bounded, $\bm{G}$-invariant set satisfying $\partial \mathcal{U} \cap \Lambda^{\bm{H}} = \emptyset$. If $(H) \in \Phi_1^t(\bm{G})$ is a \textbf{maximal} orbit type in the support of the aggregate virtual spectral flow $\mathbb{W}_{\mathcal{U}}$ and the formal dimension of its aggregate fixed-point space is non-zero, i.e.,
\[
\dim_\bc (\mathbb{W}_{\mathcal{U}})^H = \sum_{(\lambda_i,0) \in \mathcal U \cap \Lambda^{\bm H}} \left( \sum_{m \in 2\bn - 1} \sum_{j \in \mathcal{J}}  \rho_{m,j}(\lambda_i) \dim_\bc \mathcal{V}_{m,j}^H \right) \neq 0,
\]
then there exists a branch of solutions $\mathscr C$ emerging from the trivial branch at some critical point inside $\mathcal U$ that intersects the boundary: $\mathscr C \cap \partial \mathcal U \neq \emptyset$.
\end{theorem}
\begin{proof}
In the $\bm H$-fixed problem the $m=0$ sector is absent, so the stationary determinant signs in $\dim_\bc \left(\mathbb W_{\mathcal U}(K) \right)^K$ are all $+1$. Hence $\dim_\bc \left(\mathbb W_{\mathcal U}(K) \right)^K=\dim_\bc(\mathbb W_{\mathcal U})^K$ for all orbit types $(K) \in \Phi_1^t(\bm G)$.
Maximality of $(H)$ in the support of $\mathbb W_{\mathcal U}$ gives $\dim_\bc \left(\mathbb W_{\mathcal U}(L) \right)^L=0$ for all $(L)>(H)$, while the hypothesis gives $\dim_\bc \left(\mathbb W_{\mathcal U}(H) \right)^H\neq0$ such that Corollary \ref{cor:macro_escape} applies (adapted to the $\bm H$-fixed-point setting).
\end{proof}

\begin{corollary}\label{cor:spatio_temporal_holomorphic_escape}
Suppose that $\Lambda^{\bm H}$ is discrete and let $\mathcal{U} \subset \bc \times \mathscr{H}^{\bm{H}}$ be an open, bounded, $\bm{G}$-invariant set satisfying $\partial \mathcal{U} \cap \Lambda^{\bm{H}} = \emptyset$. If the non-stationary eigenvalues of $\mathscr{A}^{\bm{H}}(\lambda)$ are holomorphic in $\mathcal{U}$ and $(H) \in \Phi_1^t(\bm{G})$ is a \textbf{maximal} orbit type in the aggregate generalized kernel
\[
\bigoplus_{(\lambda_i,0) \in \mathcal{U} \cap \Lambda^{\bm H}} \ker\nolimits_{gen}\mathscr{A}^{\bm{H}}(\lambda_i),
\]
then there exists a branch of solutions $\mathscr C$ emerging from the trivial branch at some critical point inside $\mathcal{U}$ that intersects the boundary $\partial \mathcal{U}$.
\end{corollary}
\begin{proof}
By Lemma~\ref{lem:holomorphic_kernel}, each local spectral flow
$\mathbb W_{\lambda_i}^{\bm H}$ is an actual $\bm G$-representation whose
support coincides with the support of the dynamic generalized kernel of
$\mathscr A^{\bm H}(\lambda_i)$. Hence the aggregate spectral flow
$\mathbb W_{\mathcal U}$ is an actual representation whose support coincides
with the support of the aggregate generalized kernel.

Since $(H)$ is a maximal orbit type in this aggregate kernel, it is also an orbit type in the support of $\mathbb W_{\mathcal U}$. Thus, for at least one critical point $\lambda_i$ and one active $(m \geq 1)$ index $(m,j)$, one has both $\rho_{m,j}(\lambda_i)>0$ and $\dim_\bc \mathcal V_{m,j}^H>0$. Moreover all coefficients $\rho_{m,j}(\lambda_i)$ are non-negative by
Lemma~\ref{lem:holomorphic_kernel}. Therefore every term in
\[
\dim_\bc(\mathbb W_{\mathcal U})^H
=
\sum_{(\lambda_i,0)}
\sum_{m,j}
\rho_{m,j}(\lambda_i)\dim_\bc\mathcal V_{m,j}^H
\]
is non-negative, and at least one term is strictly positive. Hence
$\dim_\bc(\mathbb W_{\mathcal U})^H>0$, so Theorem~\ref{thm:spatio_temporal_macroscopic_escape}
applies.
\end{proof}

\section{Symmetric Hopf Bifurcation in a $\Gamma$-Equivariant Network}
\label{sec:application}
The abstract framework of
Sections~\ref{sec:bif_framework}--\ref{sec:spectral_flow} reduces the detection of two-parameter equivariant bifurcation at maximal orbit types to a
dimension count on the complex spectral flow representation.  We now demonstrate that this reduction yields concrete, computable results by applying it to symmetric Hopf bifurcation in a parameterized first-order ODE---a setting where the standard basic-degree pipeline is particularly unwieldy (see \cite[Chapter 11]{book-new}).
\vs
Let $\Gamma$ be a finite group acting on a Euclidean space $W := \br^N$ by permutation of coordinates:
\[
  \gamma(x_1, \ldots, x_N)
  := (x_{\gamma(1)}, \ldots, x_{\gamma(N)}).
\]
We interpret the $\Gamma$-representation $W$ as the configuration space of some networked $\Gamma$-symmetric system, modeled by the equation 
\begin{equation}\label{eq:example_system}
    \dot{x}(t) = f(\alpha, x(t)), \quad \alpha \in \br, \; x(t) \in W,
\end{equation}
where the map $f: \br \times W \to W$ satisfies the standard conditions:
\begin{enumerate}[label=($W_\arabic*$)]
    \item\label{w1} $f$ is continuous and $\Gamma$-equivariant;
    \item\label{w2} $f(\alpha, 0) = 0$ for all $\alpha \in \br$;
    \item\label{w3} the derivative $A(\alpha) := D_x f(\alpha, 0)$ exists, depends continuously on $\alpha$, and satisfies 
    \[
    \lim_{|v| \to 0} |f(\alpha,v) - A(\alpha)v|/|v| = 0.
    \]
\end{enumerate}
A $\Gamma$-equivariant vector field in this context describes dynamics which are compatible with the permutation action of $\Gamma$ on $W$. Thus synchrony subspaces, cluster subspaces, and more general fixed-point subspaces $W^H$ are invariant under the flow.  Periodic solutions born in Hopf bifurcation therefore carry both a spatial symmetry type, determined by an isotropy subgroup of $\Gamma$, and a temporal phase symmetry, represented by the $S^1$-action on periodic functions. 
\vs
To detect periodic solutions to \eqref{eq:example_system} of unknown period $p > 0$, one rescales time via a \emph{frequency parameter} $\beta := 2\pi/p$, reformulating the problem in the functional Banach space $\mathbb{E} = C_{2\pi}^1(\br; W)$ as the two-parameter operator equation
\begin{equation}\label{eq:example_operator_equation}
    \mathscr{F}(\alpha, \beta, x) := x - (L + \bm{j})^{-1}\left(\frac{1}{\beta}N_f(\alpha, \bm{j}(x)) + \bm{j}(x)\right) = 0, \quad x \in \mathbb{E},
\end{equation}
where $L: \mathbb E \to C_{2\pi}(\br; W)$ is the differential operator $L(x) := \dot{x}$, $\bm{j}$ is the compact embedding into $C_{2\pi}(\br; W)$ and $N_f: \br \times C_{2\pi}(\br; W) \to C_{2\pi}(\br; W)$ is the continuous superposition operator given by $N_f(\alpha, x) : = f(\alpha,x(t))$. Following \cite{book-new}, we notice that $\mathscr F: \br \times \br_+ \times \mathbb E \to \mathbb E$ (here $\br_+ := (0,\infty)$) is a completely continuous $G:= S^1 \times \Gamma$-equivariant field satisfying $\mathscr F(\alpha,\beta,0) = 0$ for all $(\alpha,\beta) \in \br \times \br_+$. Its linearization about the trivial branch is 
\[
\mathscr A(\alpha,\beta) := D_x \mathscr F(\alpha,\beta,0) : \mathbb E \to \mathbb E, \quad \mathscr A(\alpha,\beta)x = x - (L + \bm{j})^{-1}\left(\frac{1}{\beta}A(\alpha) \circ \bm j(x) + \bm{j}(x)\right).
\]
Moreover, if the state-space $W$ admits the $\Gamma$-isotypic decomposition $W = \bigoplus_{j=0}^r W_j$ with each $\Gamma$-isotypic component $W_j$ modeled on $\ell_j \in \bn \cup \{0\}$ copies of the corresponding irreducible $\Gamma$-representation $\operatorname{Irr}(\Gamma) = \{ \mathcal V_j \}_{j=0}^r$, then the continuous family of $\Gamma$-equivariant matrices $A: \br \to L^\Gamma(W)$ admits the block-matrix decomposition 
\[
A(\alpha) = \bigoplus_{j=0}^r A_j(\alpha), \quad  A_j(\alpha):= A(\alpha)|_{W_j}: W_j \to W_j.
\]
In turn, the phase-space $\mathbb E$ and the two-parameter family of linear, $G$-equivariant operators $\mathscr A(\alpha,\beta)$ admit the corresponding $G$-isotypic decompositions
\begin{align*}
    \begin{cases}
       \mathbb E = \overline{\bigoplus_{j=0}^r \bigoplus_{m=0}^\infty \mathbb E_{m,j}}, \quad &\mathbb E_{m,j}:= \{ \cos(mt)a + \sin(mt) b : a,b \in W_j \};  \\
       \mathscr A(\alpha,\beta) = \bigoplus_{j=0}^r \bigoplus_{m=0}^\infty \mathscr A_{m,j}(\alpha,\beta), \quad &\mathscr A_{m,j}(\alpha,\beta) := \mathscr A(\alpha,\beta)|_{\mathbb E_{m,j}} : \mathbb E_{m,j} \to \mathbb E_{m,j}.
    \end{cases}
\end{align*}
If we indicate by $\sigma(A_j(\alpha)) = \{ \zeta_{j,1}(\alpha), \ldots , \zeta_{j, \ell_j}(\alpha) \}$ the eigenvalues $\zeta_{j,k}: \br \to \bc$ of the $\Gamma$-isotypic restriction $A_j(\alpha)$, then each $G$-isotypic restriction $\mathscr A_{m,j}(\alpha,\beta)$ has the spectrum
\[
\sigma( \mathscr A_{m,j}(\alpha,\beta)) = \{ \mu_{m,j,k}(\alpha,\beta) \}, \quad \mu_{m,j,k}(\alpha,\beta) := \frac{i m \beta - \zeta_{j,k}(\alpha)}{\beta(1 + i m)}.
\]
Consequently, a trivial solution 
$(\alpha_0,\beta_0, 0) \in \br \times \br_+ \times \mathbb E$ is a \emph{stationary} critical point for \eqref{eq:example_operator_equation} if and only if for some index pair $(j,k) \in \{0,\ldots,r\} \times \{1,\ldots, \ell_j\}$, one has $\zeta_{j,k}(\alpha_0) = 0$ and a \emph{dynamic} critical point if and only if for some index triple $(m,j,k) \in \bn \times \{0,\ldots,r\} \times \{1,\ldots, \ell_j\}$, one has
\[
m \beta_0 = \operatorname{Im}(\zeta_{j,k}(\alpha_0)), \quad \text{ and } \quad \operatorname{Re}(\zeta_{j,k}(\alpha_0)) = 0.
\]
Suppose $(\lambda_0, 0) \in \bc \times \mathbb E$ is an isolated dynamic critical point (where $\lambda_0 = \alpha_0 + i\beta_0$). We wish to compute the coefficient of a maximal orbit type $(H) \in \Phi_1^t(G)$ in the local bifurcation invariant $\omega_G(\lambda_0)$ to guarantee the emergence of $H$-symmetric periodic solutions.
\vs
In the standard framework, the local bifurcation invariant $\omega_G(\lambda_0) \in A_1^t(G)$ is factored into a product of a stationary component and a dynamic component. Specifically, if we let $m_j^-(\alpha_0)$ denote the total algebraic multiplicity of negative eigenvalues of the stationary block $\mathscr A_{0,j}(\alpha_0,\beta_0)$ (equivalently the multiplicity of positive real eigenvalues of $A_j(\alpha_0)$), and we denote by $\rho_{m,j}(\lambda_0) \in \bz$ the net crossing number of the eigenvalues $\zeta_{j,k}(\alpha)$ passing through the purely imaginary critical value $im\beta_0$ (these are exactly the total winding numbers defined in \eqref{def:spectral_flow_rep}), then the local bifurcation invariant can be expressed as an $A(\Gamma)$-module product of the basic degrees $\deg_{\mathcal V_j} := \Gammadeg(-\id, B(\mathcal V_j)) \in A(\Gamma)$ and $\deg_{\mathcal V_{m,j}} \in A_1^t(G)$ (see Appendix \ref{sec:appendix_twisted_comp_form}, Corollary \ref{cor:standard_basic_degree_factorization}):
\begin{equation} \label{eq:standard_pipeline_factorization}
    \omega_G(\lambda_0) = \prod_{j=0}^r \left( \deg_{\mathcal V_j} \right)^{m_j^-(\alpha_0)} \;\cdot \sum_{m>0} \sum_{j=0}^r \rho_{m,j}(\lambda_0) \deg_{\mathcal{V}_{m,j}}.
\end{equation}

\subsection{A Concrete Spectral Configuration}
\label{sec:concrete_config}
We now specialize to a concrete scenario that illustrates both  local and global bifurcation guarantees. Consider the system
\eqref{eq:example_system} with the spatial symmetry group $\Gamma = S_4 \times \bz_2$, where $S_4$ encodes the permutation symmetry of the network and $\bz_2$ arises from the oddness of the nonlinearity.
Let the state space be $W = \mathcal{V}_2 \oplus \ell_3\mathcal{V}_3$ with multiplicity $\ell_3 \ge 4$, where $\mathcal{V}_2$ is the $2$-dimensional irreducible representation of $S_4$ and $\mathcal{V}_3$ is the standard $3$-dimensional geometric representation, both equipped with the antipodal $\bz_2$-action. Since $\mathcal{V}_3$ is absolutely irreducible, Schur's lemma implies $\operatorname{End}_\Gamma(\mathcal{V}_3) \cong \br$; the $\ell_3 \ge 4$ multiplicity ensures the restriction $A_3(\alpha)$ is a block matrix large enough to admit multiple pairs of complex conjugate eigenvalues.
\vs
We assume the linearization $A(\alpha) = A_2(\alpha) \oplus A_3(\alpha)$ satisfies the following spectral configuration at an isolated dynamic critical
point $\lambda_0 = \alpha_0 + i\beta_0$:
\begin{enumerate}[label=(\roman*)]
\item\label{spec:dynamic}
\textbf{Dynamic spectrum:} The block matrix $A_3(\alpha)$ assumes a standard normal form admitting complex conjugate eigenvalues $\zeta_{3,\pm}(\alpha) = \alpha \pm i\omega$. Exactly one positive eigenvalue branch $\zeta_{3,+}$ crosses the imaginary axis transversally at $\alpha_0$ with frequency $\beta_0 > 0$. This contributes a single crossing at the primary harmonic ($m=1$) with net crossing number $\rho_{1,3}(\lambda_0) = -1$.
  \item\label{spec:stationary}
\textbf{Stationary spectrum:} The eigenvalue $\zeta_2(\alpha_0)$ of the restriction $A_2(\alpha_0): \mathcal{V}_2 \to \mathcal{V}_2$ is strictly positive, yielding a single negative stationary mode
($m_2^-(\alpha_0) = 1$).  Apart from the critical conjugate pair on the $\mathcal V_3$-isotypic component, all remaining eigenvalues of $A_3(\alpha_0)$ are bounded away from the imaginary axis, and $A_3(\alpha_0)$ has no positive real eigenvalues  ($m_3^-(\alpha_0)=0$).
\end{enumerate}
\begin{remark}\rm
This spectral configuration arises generically in systems where the $\mathcal{V}_2$-isotypic component is damped (eigenvalues bounded away from the imaginary axis in a neighborhood of $\alpha_0$) while a Hopf instability develops on the $\mathcal{V}_3$-isotypic component. The key structural
feature is the \emph{coexistence} of a nontrivial stationary signature on $\mathcal{V}_2$ with a dynamic crossing on $\mathcal{V}_3$---precisely the
situation where the standard pipeline requires a nontrivial Burnside module product.
\end{remark}
Under these hypotheses, the local bifurcation invariant factors as
\[
  \omega_G(\lambda_0)
  = - \deg_{\mathcal{V}_2} \cdot \deg_{\mathcal{V}_{1,3}},
\]
where $\deg_{\mathcal{V}_2} \in A(\Gamma)$ and
$\deg_{\mathcal{V}_{1,3}} \in A_1^t(G)$ are the basic degrees.
We seek to detect the emergence of periodic solutions with the maximal twisted
symmetry $H = (D_4^p)^{\varphi,1} \leq S^1 \times \Gamma$, where the spatial
projection is $K=D_4^p:=D_4\times \bz_2$. If we write $D_4=\langle \gamma,\kappa \mid \gamma^4=\kappa^2=1,\;
\kappa\gamma\kappa=\gamma^{-1}\rangle$ and write the central $\bz_2$-factor multiplicatively as $\{\pm1\}$, then we can introduce the ancillary notation $D_4^d$ to indicate the diagonal copy of $D_4$ in $D_4^p:=D_4\times\mathbb Z_2$, i.e.
\[
D_4^d
:=
\{(\kappa^p\gamma^q,(-1)^q):p\in\{0,1\},\ q\in\{0,1,2,3\}\}
\leq D_4\times \bz_2.
\]
The twist homomorphism
$\varphi:D_4^p\to\{\pm1\}\leq S^1$ is then given by
$\varphi(\kappa^p\gamma^q,\pm 1)=\pm (-1)^q$
so that $\ker\varphi=D_4^d$ and $D_4^p/D_4^d\simeq \bz_2$.
\begin{remark}\rm
The condition of being fixed by the twisted subgroup $H=(D_4^p)^{\varphi,1}$ should be understood as a spatio-temporal constraint on the critical Fourier mode, rather than as a pointwise fixed-point condition for the spatial group $D_4^p$. Indeed, an element $(\varphi(g),g)\in H$ fixes a first-harmonic vector $u(t)=\operatorname{Re}(e^{it}v)$ precisely when the temporal phase $\varphi(g)\in\{\pm1\}$ compensates the spatial action of $g$ on $v$, i.e.
\[
g v=\varphi(g)^{-1}v.
\]
Thus elements in $\ker\varphi=D_4^d$ fix the spatial vector $v$ itself, while elements of $D_4^p\setminus D_4^d$ send $v$ to $-v$ and are compensated by a half-period time shift. In particular, $H$-fixed critical modes are not generally fixed pointwise by the full spatial group $D_4^p$; they are fixed by the graph of the twist homomorphism.
\end{remark}
\vs
\noi\textbf{The Basic-Degree Pipeline}\;
To extract the $(H)$-coefficient using \eqref{eq:standard_pipeline_factorization}, a researcher must first compute the basic degrees $\deg_{\mathcal{V}_2} \in A(\Gamma)$ and $\deg_{\mathcal{V}_{1,3}} \in A_1^t(G)$. Since evaluating these degrees requires resolving complete isotropy lattices, this step is typically deferred to computational software like the \texttt{EquiDeg} package \cite{GAP} in GAP:
\begin{lstlisting}[language=GAP, caption={GAP Script for Basic Degrees of $S_4 \times \bz_2$ and $S^1 \times S_4 \times \bz_2$}] 
LoadPackage("EquiDeg"); 
# Create groups SO(2), S4, and Z2
so2 := SpecialOrthogonalGroupOverReal(2);
s4 := SymmetricGroup(4);
z2 := pCyclicGroup(2);
# Generate G1 = S4 x Z2 and set abbreviated conjugacy class names
G1 := DirectProduct(s4, z2);
SetCCSsAbbrv(G1, [ "Z1", "Z2", "D1z", "D1", "Z2m", "Z1p", "Z3", "V4", "D2z", "Z4", "D2", "D1p", "D2d", "V4m", "D2p", "Z4d", "D3", "D3z", "Z3p", "V4p", "D4z", "D4d", "Z4p", "D4", "D4z", "D4hd", "D3p", "A4", "D4p", "A4p",  "S4", "S4m", "S4p"]);
# Output the character table to identify irreps
tbl := CharacterTable(G1);
# Compute basic degrees in A(S4 x Z2)
irrs1 := Irr(G1);
for i in [1..10] do
    basic_deg := BasicDegree(irrs1[i]);
    PrintFormatted("\n Computing Basic Degree in V_{} \n", i);
    Display(basic_deg);
od;
# Compute twisted basic degrees in A_1^t(S^1 x S_4 x Z_2)
G := DirectProduct(so2, G1);
irrs := Irr(G);
for i in [1..10] do
    basic_deg := TwistedBasicDegree(irrs[1,i]);
    PrintFormatted("\n Computing Basic Degree in V_1,{} \n", i);
    Display(basic_deg);
od;
\end{lstlisting}
Examining the character table output, the central $\bz_2$-class is the class denoted $2a$.  The rows $X.5$ and $X.8$ have character values $-2$ and $-3$ on this class, respectively, while $X.6$ and $X.10$ have values $+2$ and $+3$.  Hence the antipodal $\bz_2$-actions correspond to
\[
\mathcal V_2 \leftrightarrow \texttt{V\_5}, \qquad
\mathcal V_3 \leftrightarrow \texttt{V\_8}.
\]  
The basic degrees evaluate as follows:
\begin{lstlisting}[language=GAP, caption={GAP Output for the Basic Degrees $\deg_{\mathcal V_2}$ and $\deg_{\mathcal V_{1,3}}$}]
gap> bdeg := BasicDegree( irrs1[5] );
<1(8)-1(24)-1(26)+1(33)> in Brng( Group([ (1,2,3,4), (1,2), (5,6) ]) )
gap> tdeg := TwistedBasicDegree(irrs[1,8]);
<-1(1,16)-1(1,24)+1(1,43)+1(1,52)+1(1,53)+1(1,83)+1(1,89)> in A_1^t( <group> )
\end{lstlisting}
Here, the Burnside classes $(8)$, $(24)$, $(26)$ and $(33)$ denote the spatial groups $V_4 \simeq \bz_2 \times \bz_2$, $D_4$, $D_4^{\hat d}$ and $S_4^p := S_4 \times \bz_2$ respectively, while the twisted module class $(1,89)$ corresponds to the maximal twisted subgroup $(D_4^p)^{\varphi,1}$. Using a simple routine to lift the stationary element into the 0-th temporal mode of the ambient module structure, we evaluate the full  $A(\Gamma)$-module product directly:
\begin{lstlisting}[language=GAP, caption={GAP Function for $A(\Gamma)$-Module Multiplication and Local Bifurcation Invariant Evaluation}]
gap> ModuleProduct := function( a, b )
>     local mod_basis, lifted_a, i, coeff, id;
>     mod_basis := Basis( FamilyObj(b)!.ring );
>     lifted_a := Zero( FamilyObj(b)!.ring );
>     for i in [1 .. Length(a)] do
>         coeff := a!.coeffList[i];
>         id := a!.ccsIdList[i];
>         if IsList(id) then id := id[Length(id)]; fi;
>         lifted_a := lifted_a + coeff * mod_basis[0, id];
>     od;
>     return lifted_a * b;
> end;
\end{lstlisting}
To isolate the bifurcating symmetries for our target $(D_4^p)^{\varphi,1}$, we extract the component corresponding to the $(1,89)$ index from the resulting evaluation:
\begin{lstlisting}[language=GAP, caption={Isolating the Maximal Twisted Orbit Type Coefficient}]
gap> result := ModuleProduct( bdeg, tdeg );
<1(1,10)+1(1,14)+1(1,15)-1(1,16)+1(1,19)-1(1,22)-1(1,24)-1(1,28)-1(1,41)+1(1,43)+1(1,52)+1(1,53)-1(1,71)-3(1,75)+1(1,83)+1(1,89)> in A_1^t( <group> )
# Extract the coefficient of the maximal twisted orbit type
gap> Display(-result);
A_1^t( <group> ) element:
...
-1      (1,89)  (Z_2|Z_1 x D4d|D4p)
\end{lstlisting}
The output directly gives 
\[
\operatorname{coeff}^{(D_4^p)^{\varphi,1}}\left( \deg_{\mathcal V_2} \cdot \deg_{\mathcal V_{1,3}} \right) = 1 \implies \operatorname{coeff}^{(D_4^p)^{\varphi,1}}\left( \omega_G(\lambda_0) \right) = -1.
\]
\vs
\noi\textbf{The Complex Spectral Flow Approach.}\;
By contrast, the methodology of Section~\ref{sec:spectral_flow} circumvents the Burnside module arithmetic entirely by lifting the spectral data directly into the representation ring.
\vs
\noi\textbf{Step 1: The Complex Spectral Flow.}\;
The winding numbers of the critical eigenvalues around the origin in the complex plane coincide with the crossing numbers
$\rho_{m,j}(\lambda_0)$. Since only $\mathcal{V}_{1,3}$ crosses at the primary harmonic, the spectral flow representation reduces to
\[
  \mathbb{W}_{\lambda_0}
  = \sum_{m > 0} \sum_{j} \rho_{m,j}(\lambda_0)\,\mathcal{V}_{m,j}
  = -\mathcal{V}_{1,3}.
\]
\noi\textbf{Step 2: The Stationary Sign.}\;
The coefficient formula~\eqref{eq:universal_formula} requires the sign of the real determinant of the stationary linearization restricted to the $H$-fixed-point space. Since the spatial projection of $H$ is $K=D_4^p$ and $\mathcal V_2$ carries the antipodal $\bz_2$-action, one has
\[
\mathcal V_2^K=\{0\}.
\]
Hence the stationary determinant restricted to the $H$-fixed space is $+1$.

\vs
\noi\textbf{Step 3: The Formal Complex Dimension.}\;
The twisted orbit type $H=(D_4^p)^{\varphi,1}$ with $\ker\varphi=D_4^d$ fixes a $1$-complex-dimensional subspace of $\mathcal V_{1,3}$. To see this, notice that a vector $v \in \mathcal V_{1,3}$ is $H$-fixed if $gv = \varphi(g)^{-1}v$ for all $g \in D_4^p$. The $\bz_2$-valued character $\varphi$ identifies an index-two subspace of $\mathcal V_3$ on which $D_4^d$ acts trivially and $D_4^p \setminus D_4^d$ acts as $-1$, which by inspection of the character table of $D_4^p$ on $\mathcal V_3$ is one-dimensional over $\bc$. Since $\rho_{1,3}(\lambda_0)=-1$, we obtain
\[
\dim_\bc \mathbb W_{\lambda_0}^H=-1,
\qquad |W(H)/S^1|=1.
\]

\vs
\noi\textbf{Step 4: The Bifurcation Coefficient.}\;
Applying the dimension formula~\eqref{eq:universal_formula}, we obtain
\[
\operatorname{coeff}^{H}(\omega_G(\lambda_0))
  = (+1)\cdot \frac{-1}{1} = -1.
\]
This agrees with the signed GAP evaluation of the local bifurcation invariant. The spectator mode $\mathcal{V}_2$ enters only through a single determinant sign; the basic degree
$\deg_{\mathcal{V}_2}$, its internal Burnside structure, and the full module product are never invoked.

\subsubsection{Global Bifurcation via the Aggregate Flow}
To illustrate how the spectral flow scales to global continuation, we extend the scenario to a bounded parameter domain $\mathcal{U} \subset \br \times \br_+ \times \mathbb{E}$ containing two critical points $\Lambda \cap \mathcal{U} = \{(\lambda_1,0),(\lambda_2,0)\}$
with the following spectral configurations:
\begin{table}[h]
\centering
\renewcommand{\arraystretch}{1.3}
\begin{tabular}{lcc}
\hline
& $\lambda_1$ & $\lambda_2$ \\
\hline
Negative stationary modes on $\mathcal{V}_2$ & 1 & 0 \\
Crossing number $\rho_{1,3}$ & -1 & -2 \\
Local spectral flow & $-\mathcal{V}_{1,3}$ & $-2\mathcal{V}_{1,3}$ \\
\hline
\end{tabular}
\caption{Spectral Configurations for Global Bifurcation Scenario}
\end{table}
\vs
\noi\textbf{The Basic-Degree Pipeline}\;
The aggregate invariant
$\omega_G(\lambda_1) + \omega_G(\lambda_2)$
requires re-evaluating the full module arithmetic.
\begin{lstlisting}[language=GAP, caption={Evaluating the Aggregate Module Product for Global Bifurcation}]
# Compute the aggregate invariant
gap> aggregate := -result - 2 * tdeg;
<-1(1,10)-1(1,14)-1(1,15)+3(1,16)-1(1,19)+1(1,22)+3(1,24)+1(1,28)+1(1,41)-3(1,43)-3(1,52)-3(1,53)+1(1,71)+3(1,75)-3(1,83)-3(1,89)> in A_1^t( <group> )
# Extract the coefficient of the maximal twisted orbit type
gap> Display(aggregate);
A_1^t( <group> ) element:
...
-3       (1,89)  Z_2|Z_1 x D4d|D4p
\end{lstlisting}
The aggregate coefficient evaluates to $\operatorname{coeff}^{(D_4^p)^{\varphi,1}}(\omega_G(\lambda_1) + \omega_G(\lambda_2))=-3 \neq 0$. By Corollary \ref{cor:macro_escape} to the Rabinowitz alternative (Theorem~\ref{thm:rab_alt}), the non-vanishing of this coefficient guarantees that a continuous branch of $(D_4^p)^{\varphi,1}$-symmetric solutions must escape the boundary $\partial \mathcal{U}$. However, determining this required a re-evaluation of the complete module arithmetic.
\vs
\noi\textbf{The Complex Spectral Flow Approach.}\;
The aggregate virtual flow is the direct sum of the local flows:
\[
  \mathbb{W}_{\mathcal{U}}
  = \mathbb{W}_{\lambda_1} + \mathbb{W}_{\lambda_2}
  = -\mathcal{V}_{1,3} - 2\mathcal{V}_{1,3}
  = -3\mathcal{V}_{1,3}.
\]
Since $(D_4^p)^{\varphi,1}$ fixes a $1$-complex-dimensional subspace of $\mathcal{V}_{1,3}$, the aggregate formal dimension is
\[
\dim_\bc(\mathbb{W}_{\mathcal{U}})^{(D_4^p)^{\varphi,1}}
  = -3 \times 1 = -3 \neq 0.
\]
By Theorem~\ref{thm:macroscopic_escape}, the $(D_4^p)^{\varphi,1}$-symmetric branch must exit $\mathcal{U}$---regardless of whether the individual local invariants $\omega_G(\lambda_1)$ and $\omega_G(\lambda_2)$ cancel at other orbit types. The entire global verification reduces to a scalar multiplication.

\section{The Complex Ginzburg--Landau Equation on Compact Manifolds}
\label{sec:CGL_application}
The preceding example demonstrated the spectral flow's computational advantage for a non-abelian symmetry group, where the standard pipeline requires nontrivial Burnside module arithmetic. We now  turn to a qualitatively different setting: a class of problems whose eigenvalues are \emph{universally holomorphic} in the bifurcation parameters, regardless of the underlying geometry or spatial symmetry group. The holomorphic specialization (Corollary~\ref{cor:holomorphic_universal_guarantee}) then eliminates all Jacobian sign verifications simultaneously, reducing the entire bifurcation analysis to a structural observation about 
the parameter dependence.
\vs
Consider the complex Ginzburg--Landau equation (CGLe) on a compact Riemannian manifold $\Omega$ with boundary that is invariant under the isometric action of a compact Lie group $\Gamma$. The evolution of the complex amplitude $\psi(t,x) \in \bc$ is governed by the PDE:
\begin{align} \label{eq:CGL_manifold}
\begin{cases}
\partial_t \psi = (1+i\eta)\Delta_\Omega \psi 
  + \alpha\psi + f(\psi),  \\
\frac{\partial \psi}{\partial n}|_{\partial \Omega} \equiv 0,
\end{cases}
\end{align}
where $\Delta_\Omega$ is the Laplace--Beltrami operator subject to Neumann boundary conditions, $\alpha \in \br$ is the bifurcation parameter, 
$\eta \in \br$ is the dispersion coefficient, and $f: \bc \to \bc$ is a continuous map satisfying the conditions
\begin{enumerate}[label=($F_\arabic*$)]
    \item\label{f1} $f(e^{i\varphi}u) = e^{i\varphi}f(u)$ for all $u \in \mathbb{C}$ and $\varphi \in S^1$.
    \item\label{f2} $f(u) = o(|u|)$ as $u \to 0$.
    \item\label{f3} There exist constants $a, b > 0$ and $c \in (0,1)$ such that $|f(u)| \le a|u|^c + b$ for all $u \in \mathbb{C}$.
\end{enumerate}
Condition \ref{f1} ensures that the system admits the natural $S^1$-phase symmetry, condition \ref{f2} guarantees that the Fréchet derivative of the nonlinear remainder strictly vanishes at the origin, allowing the trivial branch to exist and ensuring the linearization is governed entirely by the linear terms, and condition \ref{f3} provides the necessary a priori bounds for an application of the equivariant Leray-Schauder degree.
\vs
We seek relative equilibria of the form $\psi(t,x) = e^{-i\beta t}\exp(\xi t) \cdot u(x)$, where $\beta \in \br$ is the temporal frequency and $\xi \in \mathfrak{g}$ is an element of the Lie algebra of $\Gamma$ governing the spatial drift. Substituting the ansatz transforms~\eqref{eq:CGL_manifold} into the stationary equation
\begin{align} \label{eq:CGL_stationary_manifold}
\begin{cases}
    -(1+i\eta)\Delta_\Omega u + \xi_* u = \lambda u + f(u),
\\
\frac{\partial u}{\partial n}|_{\partial \Omega} \equiv 0,
\end{cases}
\end{align}
where $\lambda := \alpha + i\beta \in \bc$ and $\xi_* u := \tfrac{d}{dt}|_{t=0}\, u(\exp(-\xi t) \cdot x)$ is the infinitesimal generator of the drift action. The symmetry group of the stationary equation \eqref{eq:CGL_stationary_manifold} is $G = S^1 \times C_\Gamma(\xi)$, where the first factor acts by temporal phase shift and the second is the centralizer of the drift in $\Gamma$. 
\vs
To formulate this problem topologically, we consider the Sobolev space $\mathscr{H} := \{ u \in H^2(\Omega; \bc) : \frac{\partial u}{\partial n}\big|_{\partial \Omega} = 0 \}$ encoding the Neumann boundary conditions. Since the shifted Laplacian $-\Delta_\Omega + I$ is a linear isomorphism from $\mathscr{H}$ to $L^2(\Omega; \bc)$, equation \eqref{eq:CGL_stationary_manifold} can be recast as a fixed-point problem for the completely continuous $G$-equivariant field $\mathscr{F}: \bc \times \mathscr{H} \to \mathscr{H}$ (see Garcia-Azpeitia, Ghanem, and Krawcewicz \cite{GGK_NODEA} for the precise Sobolev embedding arguments) given by
\begin{align*} 
\mathscr{F}(\lambda, u) := u - (-\Delta_\Omega + I)^{-1} \left( u + \frac{\lambda u - \xi_* u + f(u)}{1+i\eta} \right).
\end{align*}
The key structural observation is that the complex parameter $\lambda$ enters the linearization only through an affine spectral shift. Indeed, the linearization of the fixed-point field at the trivial solution is
\[
D_u\mathscr F(\lambda,0)
= I-(-\Delta_\Omega+I)^{-1}
\left(I+\frac{\lambda I-\xi_*}{1+i\eta}\right).
\]
Since $-\Delta_\Omega+I$ is an isomorphism, the zeros of $D_u\mathscr F(\lambda,0)$ are equivalently described by the spectral pencil
\begin{equation}\label{eq:CGL_linearization_manifold}
\mathscr A(\lambda)
:=
\mathcal L-\lambda I,
\qquad
\mathcal L:=-(1+i\eta)\Delta_\Omega+\xi_*.
\end{equation}
Thus the bifurcation spectrum is governed by the affine shift
$\mathcal L-\lambda I$. If we denote by $\{\kappa_j\}_{j \in \mathcal{K}}$ the eigenvalues of the spatial operator $\mathcal{L} := -(1+i\eta)\Delta_\Omega + \xi_*$, which are 
discrete and of finite multiplicity since $\Omega$ is compact (the resolvent of $\Delta_\Omega$ is compact), then the spectrum of $\mathscr{A}(\lambda)$ evaluates to $\mu_j(\lambda) = \kappa_j - \lambda$. Since $\lambda$ enters as a linear shift, each eigenvalue $\mu_j(\lambda)$ is \textbf{holomorphic} in $\lambda$---a property that holds 
universally, regardless of the geometry of $\Omega$, the structure of $\Gamma$, or the choice of drift $\xi$.
\vs
Since the critical set $\Lambda := \{ (\lambda_j,0) \in \bc \times \mathscr H: \lambda_j = \kappa_j \}$ is discrete, the holomorphic specialization of the spectral flow framework applies directly. Corollary~\ref{cor:holomorphic_escape} then guarantees the macroscopic emergence of symmetric patterns entirely from the linear spectrum:
\begin{theorem}
\label{thm:CGL_universal_escape}
For any maximal orbit type $(H)$ in the generalized kernel of $\mathcal{L} - \lambda_j I$, there exists an unbounded continuum of $H$-symmetric relative equilibria emerging from the trivial branch at $(\lambda_j,0)$.
\end{theorem}
\begin{proof}
The phase space $\mathscr{H}$ consists entirely of complex-valued functions transforming under the $1$-folded phase shift $e^{i\varphi} u$, so its $G$-isotypic decomposition admits no stationary ($m=0$) modes. Consequently, the stationary determinant is identically $+1$ everywhere, and the constant-sign hypothesis of Corollary~\ref{cor:holomorphic_escape} is vacuously satisfied. 

Since the eigenvalues $\mu_j(\lambda) = \kappa_j - \lambda$ are holomorphic, Lemma~\ref{lem:holomorphic_kernel} implies that the spectral flow representation $\mathbb{W}_{\lambda_j}$ shares the exact support of the dynamic generalized kernel, with every winding number being non-negative. 
Since $(H)$ is maximal in this kernel, its formal dimension $\dim_\bc \mathbb{W}_{\lambda_j}^H > 0$ is positive. Therefore, the aggregate signed fixed-point dimension $\dim_\bc \left(\mathbb W_{\mathcal U}(H) \right)^H$ over any bounded parameter domain $\mathcal{U}$ containing the critical point evaluates to a strictly positive sum and cannot vanish. 

The macroscopic escape of the branch then follows immediately from Corollary~\ref{cor:holomorphic_escape} and, since $\mathcal U$ can be taken to be an arbitrarily large bounded domain containing $(\lambda_j, 0)$, the branch cannot be contained in any bounded subset of $\bc \times \mathscr H$, yielding an unbounded continuum.
\end{proof}

We illustrate the universality of this conclusion with two specializations to canonical geometries.
\vs
\noi\textbf{Specialization I: Rotating Spiral Waves on the Planar Disc.}\;
Let $\Omega = D \subset \br^2$ be the unit disc with $\Gamma = SO(2)$, and set the drift to a rotation $\xi_* = \omega\partial_\theta$ such that our relative equilibria take the form $\psi(t,r,\theta) = e^{-i\beta t}u(r,\theta+ \omega t)$. The centralizer is $C_{SO(2)}(\xi) = SO(2)$, so the symmetry group of our problem becomes $G = S^1 \times SO(2) \simeq \bt^2$. The phase space decomposes into isotypic components $\mathscr{E}_{m,n}$ modeled on the irreducible $\bt^2$-representation $\mathcal U_m \otimes \mathcal U_1 \simeq \bc$ and indexed by $(m,n) \in \bz \times \bn$, where $n$ labels the non-negative roots $\sqrt{s_{|m|,n}}$ of $J'_{|m|}$ (the derivative of the $|m|$-th Bessel function of the first kind), arising from the Neumann eigenvalue problem on $D$. The maximal orbit types in $\mathscr H$ are the spiral wave symmetries $H_m := \{(e^{im\varphi}, e^{-i\varphi})\} \leq \bt^2$ arising as the only non-trivial isotropy groups in each irreducible $\bt^2$-representation $\mathcal U_m \otimes \mathcal U_1$.
\vs
Evaluating the spatial operator $\mathcal{L} = -(1+i\eta)\Delta_D + \omega \partial_\theta$ on the component $\mathscr{E}_{m,n}$ yields the eigenvalues $\kappa_{m,n} = (1+i\eta)s_{|m|,n} + i \omega m$. The critical points are thus the isolated parameter values $\lambda_{m,n} := \kappa_{m,n}$ at which the linearized operator $\mathscr{A}(\lambda_{m,n})$ becomes singular.
\vs
In a recent work~\cite{GGK_NODEA}, Garcia-Azpeitia et al. used the standard $\bt^2$-equivariant degree to establish that every critical point $(\lambda_{m,n},0)$ is the branching point for an unbounded branch of $|m|$-armed spiral wave solutions under general conditions on 
the nonlinearity. Since $\bt^2$ is abelian, the Burnside module arithmetic is trivial, but the standard pipeline still required analytically deriving the $2 \times 2$ Jacobian matrix of the 
complex mapping $\mu_{m,n}(\alpha,\beta) := \kappa_{m,n} - (\alpha + i \beta)$ and verifying that its determinant is strictly positive. Under the spectral flow framework, this computation is rendered unnecessary by the 
universal holomorphic structure of~\eqref{eq:CGL_linearization_manifold}: the existence of unbounded branches of spiral waves is an immediate consequence of the holomorphic specialization.
\vs
\begin{remark}\rm
The results of~\cite{GGK_NODEA} and those obtained here via the spectral flow are equivalent: both establish the existence of unbounded branches of $|m|$-armed spiral waves under the same conditions on the nonlinearity. The difference is methodological. In~\cite{GGK_NODEA}, non-cancellation across critical points is 
verified by computing $\omega_{\bt^2}(\lambda_{m,n}) = +1 \cdot (H_m)$ at each 
individual critical point. In the spectral flow framework, this non-cancellation is a structural consequence of holomorphic eigenvalue dependence: Cauchy's Argument Principle forces every winding number to be non-negative, making cancellation impossible a priori.
\end{remark}

\vs
\noi\textbf{Specialization II: Standing Waves on the Sphere.}\;
Let $\Omega = S^2 \subset \br^3$ with $\Gamma = O(3)$, and set $\xi = 0$ (standing waves) such that the symmetry group becomes $G = S^1 \times O(3)$. The phase space decomposes into isotypic components corresponding to spatial spherical harmonics of degree $l \in \bn_0$, with each $G$-isotypic component modeled on the irreducible representation $\mathcal{U}_1 \otimes \mathcal{W}_l$, where  $\mathcal{U}_1 \simeq \bc$ is the $1$-folded irreducible $S^1$-representation and $\mathcal{W}_l$ is 
the $(2l+1)$-dimensional irreducible $O(3)$-representation associated with spherical harmonics of degree $l$.
\vs
Since the Laplace-Beltrami operator acts on the spherical harmonics $Y_l^k$ (for $-l \leq k \leq l$) as $-\Delta_{S^2} Y_l^k = l(l+1) Y_l^k$, the spatial operator $\mathcal{L} = -(1+i\eta)\Delta_{S^2}$ evaluates on the $l$-th isotypic component exactly to the highly degenerate eigenvalues $\kappa_l = (1+i\eta)l(l+1)$. These are distinct for distinct $l$, so the associated critical points $\lambda_l = \kappa_l$ are isolated. 
\vs
The isotropy lattice
$\Phi_1^t(S^1 \times O(3);\, \mathcal U_1 \otimes \mathcal W_l)$
is remarkably rich.  In the standard notation of the twisted equivariant
degree~\cite{AED}, a twisted subgroup is written
$K^{\varphi,m} = \{(z,g)\in S^1\times K:\ z^m=\varphi(g)\}$ where $K\leq O(3)$ and $\varphi:K\to S^1$ is a continuous character.  Since the $S^1$-gauge action on $\mathscr H$ is $1$-folded in the present specialization,
any isotropy group fixing a non-zero vector in
$\mathcal U_1\otimes\mathcal W_l$ has $m=1$.  Thus the relevant twisted
subgroups are graph subgroups
\[
K^{\varphi,1}
=
\{(\varphi(g),g):g\in K\}.
\]
When $\varphi$ is trivial, we suppress it and write simply $K$.  When
$\varphi$ has image $\{\pm1\}$, we use the superscripts of
Ihrig and Golubitsky~\cite{IhrigGolubitsky1984} to record the corresponding
index-two kernel. In particular, if $L\leq K$ is an index-two subgroup, let
$\varphi_{K,L}:K\to\{\pm1\}$ be the quotient character with
$\ker\varphi_{K,L}=L$.  Then
\[
K^{\varphi_{K,L},1}
=
\{(\varphi_{K,L}(g),g):g\in K\}
\leq S^1\times O(3).
\]
We write $O(2)^-$ for the case $K=O(2)$ and $L=SO(2)$, and
$\mathbb O^-$ for the case $K=\mathbb O$ and $L=\mathbb T$, where
$\mathbb T$ is the rotational tetrahedral subgroup of the rotational
octahedral group $\mathbb O$.

The dihedral superscript requires a separate convention.  Let
$D_{2n} := \langle \rho,\tau:\rho^{2n}=\tau^2=e,\ \tau\rho\tau=\rho^{-1}\rangle\leq SO(3)$ denote the rotational dihedral group generated by a rotation
$\rho$ of order $2n$ about a fixed axis and a half-turn $\tau$ about a
perpendicular axis, containing the index-two subgroup $D_n:=\langle \rho^2,\tau\rangle \leq D_{2n}$. Let $\varphi_d:D_{2n}\to\{\pm1\}$ be the character $\varphi_d(\rho^a\tau^b) :=(-1)^a$. Following the class-III notation of~\cite{IhrigGolubitsky1984}, we denote the
corresponding twisted subgroup of $S^1\times O(3)$ by
\[
D_{2n}^d := (D_{2n})^{\varphi_d,1} = \{((-1)^a,\rho^a\tau^b):0\leq a<2n,\ b\in\{0,1\}\}.
\]
\vs
Thus the superscript $d$ records the diagonal index-two choice
$\ker\varphi_d=D_n$.
\vs
With these conventions, the maximal orbit types in $\Phi_1^t(S^1 \times O(3);\, \mathcal U_1 \otimes \mathcal W_l)$ are determined by the parity of $l$ and by the fundamental invariants of the Platonic groups (see Ihrig and Golubitsky~\cite{IhrigGolubitsky1984}).  Let
$Z_2^c=\{\pm I\}$ denote the central subgroup of $O(3)$.
\begin{itemize}
\item For $l=1$, the only maximal symmetry is $O(2)^-$.

\item For $l=2$, the only maximal symmetry is $O(2)\oplus Z_2^c$.

\item For $l\in\{4,8,14\}$, the maximal symmetries are
$O(2)\oplus Z_2^c$ and $\mathbb O\oplus Z_2^c$.

\item For all other even $l$, the maximal symmetries are
$O(2)\oplus Z_2^c$, $\mathbb O\oplus Z_2^c$, and
$\mathbb I\oplus Z_2^c$.

\item For $l\in\{3,7,11\}$, the maximal symmetries are
$O(2)^-$, $\mathbb O^-$, and $D_{2n}^d$ for $\frac{l}{2}<n\leq l$.

\item For $l=5$, the maximal symmetries are
$O(2)^-$, $D_6^d$, and $D_8^d$.

\item For $l\in\{9,13,17,19,23,29\}$, the maximal symmetries are
$O(2)^-$, $\mathbb O^-$, $\mathbb O$, and $D_{2n}^d$ for $\frac{l}{2}<n\leq l$.

\item For all other odd $l$, the maximal symmetries are
$O(2)^-$, $\mathbb O^-$, $\mathbb O$, $\mathbb I$, and
$D_{2n}^d$ for $ \frac{l}{2}<n\leq l$.
\end{itemize}
\begin{remark}
The exceptional low-degree omissions are governed by the invariant theory of the Platonic groups. For even $l$, spatial inversion acts trivially on $\mathcal W_l$, so the maximal spatial isotropy groups contain the central factor $Z_2^c$. The even icosahedral invariants occur in degrees generated by
$6b+10c$, which excludes precisely the exceptional even degrees $\{2,4,8,14\}$ from the icosahedral list. For odd $l$, spatial inversion acts as $-1$, so non-trivial graph twists are possible. The first odd
icosahedral invariant occurs in degree $15$, and the higher odd icosahedral degrees are generated by $15+6b+10c$; this accounts for the absence of $\mathbb I$ at $l=9,13,17,19,23,29$. The dihedral family $D_{2n}^d$ appears
through the corresponding diagonal index-two character and is present for odd
$l$ exactly in the range $\frac l2<n\leq l$, with the low-degree exception $l=5$ giving only $D_6^d$ and $D_8^d$.
\end{remark}
Under the spectral flow framework, because the eigenvalues $\mu_l(\lambda) = (1+i\eta)\,l(l+1) - \lambda$ are unconditionally holomorphic, 
\emph{the global bifurcation analysis is exactly as straightforward as it was for the planar disc}: the formal dimension of any maximal orbit type in the kernel is unconditionally positive, and macroscopic escape follows without any Burnside ring computation. The algebraic complexity of the 
isotropy lattice of $O(3)$ is entirely bypassed.

\section*{Acknowledgements}
The author would like to thank Prof. Wiesław Krawcewicz and Prof. Sławomir Rybicki for their valuable insights and helpful conversations regarding the equivariant spectral flow, which significantly contributed to the development of this work.
\newpage
\appendix
\section{The $S^1$-Equivariant Degree} \label{sec:appendix_S1}
\noi{\bf  Equivariant notation.}
Let $\Gamma$ be a compact Lie group. For any subgroup  $H \leq \Gamma$, we denote by $(H)$ its conjugacy class,
by $N(H)$ its normalizer by $W(H):=N(H)/H$ its Weyl group in $\Gamma$. The set of all subgroup conjugacy classes in $\Gamma$ is denoted by $\Phi(\Gamma):=\{(H): H\le \Gamma\}$ and has a natural partial order defined as follows
\[
(H)\leq (K) \iff \exists_{ g\in \Gamma}\;\;gHg^{-1}\leq K.
\]
As is possible with any partially ordered set, we extend the natural order over $\Phi(\Gamma)$ to a total order, which we indicate by $<$ to differentiate the two relations. Moreover, we put $\Phi_n (\Gamma):= \{ (H) \in \Phi(\Gamma) \; : \; \dim W(H) = n\}$ and, for any $(H),(K) \in \Phi_n(\Gamma)$, we denote by $n(H,K)$ the number of subgroups $\tilde K \leq \Gamma$ with $\tilde K \in (K)$ and $H \leq \tilde K$.
\vs
Given a $\Gamma$-space $X$ with an element $x \in X$, we denote by
$\Gamma_{x} :=\{g\in \Gamma:gx=x\}$ the {\it isotropy group} of $x$
and we call $(\Gamma_{x}) \in \Phi(\Gamma)$  the {\it orbit type} of $x \in X$. Put $\Phi(\Gamma,X) := \{(H) \in \Phi_n(\Gamma)  : 
(H) = (\Gamma_x) \; \text{for some $x \in X$}\}$ and  $\Phi_n(\Gamma,X):= \Phi(\Gamma,X) \cap \Phi_n(\Gamma)$. For a subgroup $H\leq \Gamma$, the subspace $
X^{H} :=\{x\in X:\Gamma_{x}\geq H\}$ is called the {\it $H$-fixed-point subspace} of $X$. If $Y$ is another $\Gamma$-space, then a continuous map $f : X \to Y$ is said to be {\it $\Gamma$-equivariant} if $f(gx) = gf(x)$ for each $x \in X$ and $g \in \Gamma$. 
\vs
\noi{\bf Axioms of the $S^1$-Equivariant Degree.}
Let $W$ be an orthogonal $S^1$-representation with an open bounded $S^1$-invariant set $\Om \subset \br \times W$ (we assume that $S^1$ acts trivially on $\br$). An $S^1$-equivariant map $f: \br \times W \rightarrow W$ is said to be {\it $\Om$-admissible} if $f^{-1}(0) \cap \partial \Om = \emptyset$, in which case the pair $(f,\Om)$ is called an {\it admissible $S^1$-pair in $W$}. We denote by $\mathcal M_1^{S^1}(\br \times W)$ the set of all admissible $S^1$-pairs in $W$ and by $\mathcal M_1^{S^1}$ the set of {\it all} admissible $S^1$-pairs, defined by taking a union over all orthogonal $S^1$-representations as follows
\[
\mathcal M_1^{S^1}:= \bigcup\limits_W \mathcal M_1^{S^1}(\br \times W).
\]
If $W$ is a complex $S^1$-representation, then any complex linear map $W \rightarrow W$ is $S^1$-equivariant. Moreover, for any $w \in W$ one has the following two possibilities
    \[
    G_w = \begin{cases}
        \bz_k \quad & \text{ for some } k \in \bz; \\
        S^1.
    \end{cases}
    \]
With this in mind, we denote by  $\Phi(S^1)$ the set of all subgroup conjugacy classes in $S^1$ and by $\Phi_1(S^1)$ the subgroup conjugacy classes in $S^1$ with $1$-dimensional Weyl groups, i.e.
\[
\Phi_1(S^1) := \{ (\bz_k) : k \in \bn \}.
\]
We denote by $A_1(S^1)$ the free $\bz$-module generated by the set $\Phi_1(S^1) = \{ (\bz_k) \in \Phi(S^1): k \in \bn \}$. Any element $a \in A_1(S^1) = \bz[\Phi_1(S^1)]$ can be represented by a formal sum
\[
a = \sum_{k \in \bn} n_k(\bz_k), \quad n_k \in \bz,
\]
where $n_k = 0$ for almost every $k \in \bn$. We are now in a position to present an axiomatic definition of the $S^1$-equivariant degree.
\begin{theorem} \rm
    There exists a function $\s1deg: \mathscr M_1^{S^1} \rightarrow A_1(S^1)$ satisfying the properties:
    \begin{enumerate}[label=($S_\arabic*$)]
        \item\label{s1} {\bf (Additivity)}: If for any two $S^1$-invariant, disjoint subsets $\Om_1,\Om_2 \subset \Om$, one has $f^{-1}(0) \cap  \Om \subset \Om_1 \cup \Om_2$, then
        \[
        \s1deg(f,\Om) = \s1deg(f,\Om_1) + \s1deg(f,\Om_2).
        \]
        \item\label{s2} {\bf (Homotopy)}: If $h:[0,1] \times \br \times W \rightarrow W$ is an $\Om$-admissible $S^1$-equivariant homotopy, then
        \[
        \s1deg(h_0,\Om) = \s1deg(h_1,\Om),
        \]
        where $h_t: \br \times W \rightarrow W$ is used to indicate the map $h_t(\cdot, \cdot) := h(t, \cdot,\cdot)$.
        \item\label{s3} {\bf (Normalization)}: If $f$ is regular normal in $\Om$ and $f^{-1}(0) \cap \Om = S^1(w_0)$ for some $w_0 \in \Om$, then one has
    \[
        \s1deg(f,\Om) = \begin{cases}
            \rho_0 (\bz_l) \quad & \text{if } S^1_{w_0} = \bz_l \text{ for some } l \in \bn; \\
             0  \quad & \text{if } S^1_{w_0} = S^1, 
        \end{cases}
        \]
        where $\rho_0 := \operatorname{sign}\det(\left.Df(w_0)\right|_{S_{w_0}})$ (here $S_{w_0} \subset \br \times W$ denotes the slice to the orbit $S^1(w_0)$ at $w_0$, i.e. the orthogonal subspace in $\br \times W$ to the orbit $S^1(w_0)$ at $w_0$).
    \end{enumerate}
\end{theorem}
The map $\s1deg$ satisfies a number of additional properties which can be derived from the axioms \ref{s1}--\ref{s3}, including the existence property:
\begin{enumerate}[label=($S_\arabic*$), resume]
    \item\label{s4} \textbf{(Existence)}: If there is an orbit type $(\bz_k) \in \Phi_1(S^1)$ for which $\operatorname{coeff}^{(\bz_k)}(\s1deg(f,\Om)) \neq 0$, then there exists $x \in \Om$ such that $f(x) = 0$ and $S^1_x \ge \bz_k$;
\end{enumerate}
and the following product property:
\begin{lemma} \rm \label{lemm:s1_degree_product_property}
Let $W$ be a finite dimensional space on which $S^1$ acts trivially with an open bounded set $D \subset W$ and let $\phi: W \to W$ be a continuous function satisfying $\phi^{-1}(0) \cap \partial D = \emptyset$ such that the local Brouwer degree $\deg(\phi,D) \in \bz$ is well-defined. Then for any admissible $S^1$-pair $(f,\Om) \in \mathcal M^{S^1}_1$, one has $(f \times \phi, \Om \times D) \in \mathcal M^{G}_1$ and
\[
\s1deg(f \times \phi, \Om \times D) = \deg(\phi,D) \s1deg(f,\Om).
\]
\end{lemma}
There is a well-defined degree, called the {\it cumulative $S^1$-degree}, associated with the $S^1$-degree and whose role in the construction of the twisted equivariant degree is analogous to the role played by the local Brouwer degree in the equivariant Brouwer degree.
\begin{definition} \rm
The cumulative $S^1$-degree is the function $\deg_{S^1}:\mathscr M_1^{S^1} \rightarrow \bz$, given by
\[
\deg_{S^1}(f,\Om) := \sum_{s > 0} d_s, \quad d_s := \operatorname{coeff}^{(\bz_s)}(\s1deg(f,\Om)).
\]
\end{definition}
\begin{remark} \rm
    The cumulative $S^1$-degree is a total algebraic count of the non-trivial $S^1$ orbits in the solution set $f^{-1}(0) \cap \Om$ and satisfies the three degree axioms \ref{s1}--\ref{s3}, the existence property \ref{s4}, and also the product property described in Lemma \ref{lemm:s1_degree_product_property} with $\s1deg$ and $A^1(S^1)$ replaced by $\deg_{S^1}$ and $\bz$, respectively.
\end{remark}
The following computational tool reduces the $S^1$-equivariant degree of a complemented complex-linear operator on an $S^1$-isotypic space to the winding number of its complex determinant. It is the key mechanism by which the spectral data of the linearization interfaces with the equivariant degree in the main text.
\begin{lemma} \rm \label{lemm:determinantal_reduction}
Let $W$ be a an orthogonal $S^1$
representation modeled on a non-trivial irreducible $S^1$-representation $\mathcal U_m \simeq \bc$ with finite multiplicity $N$ (i.e. $W \simeq N \cdot \mathcal U_m$) and let $A\colon \overline{B}_\varepsilon(\lambda_0) \to M_N(\bc)$ be a continuous family of complex matrices satisfying $A(\lambda) \in GL_N(\bc)$ for all $\lambda \in \partial B_\varepsilon(\lambda_0)$. For any auxiliary function $\Theta$  on an isolating cylinder $\mathscr O := \{(\lambda,u) \in \bc \times W : |\lambda - \lambda_0| < \varepsilon,\; \|u\| < \delta\}$ (see Section~\ref{sec:bif_framework}), one has
\begin{equation}\label{eq:det_reduction}
\s1deg\bigl(A_\Theta, \mathscr O \bigr)
= \deg\bigl(\det\nolimits_\bc A,\; B_\varepsilon(\lambda_0)\bigr)
  \cdot (\bz_m).
\end{equation}
where $A_\Theta(\lambda) := (\Theta,\, A(\lambda))$. In particular, the cumulative $S^1$-degree evaluates to the winding number coefficient
\[
\deg_{S^1}\bigl(A_\Theta, \mathscr O \bigr)
= \deg\bigl(\det\nolimits_\bc A,\; B_\varepsilon(\lambda_0)\bigr).
\]
\end{lemma}
\begin{proof}
Write $d(\lambda) := \det_\bc A(\lambda)$ and define the block-diagonal matrix map $D: \bc \to M_N(\bc)$ given by
$D(\lambda) := \operatorname{diag}(d(\lambda),1,\dots,1)$. By construction, $\det_\bc D(\lambda) = d(\lambda) = \det_\bc A(\lambda)$ for every $\lambda \in \overline{B}_\varepsilon(\lambda_0)$.
\vs
\noi\textbf{Step 1: Boundary homotopy in $GL_N(\bc)$.}
On the boundary circle $\partial B_\varepsilon(\lambda_0)$, both $A(\lambda)$
and $D(\lambda)$ lie in $GL_N(\bc)$ and share the same determinant.
Since $GL_N(\bc)$ deformation retracts onto $U(N)$ (via Gram--Schmidt
orthonormalization), and the determinant map $\det\colon U(N) \to S^1$ induces
the isomorphism $\pi_1(GL_N(\bc)) \cong \pi_1(U(N)) \cong \bz$, two loops in
$GL_N(\bc)$ with equal determinant winding numbers are freely homotopic.
Therefore, there exists a continuous homotopy
$h_t\colon \partial B_\varepsilon(\lambda_0) \to GL_N(\bc)$, $t \in [0,1]$,
satisfying $h_0 = A|_{\partial B_\varepsilon}$ and
$h_1 = D|_{\partial B_\varepsilon}$.
\vs
\noi\textbf{Step 2: Extension to the disk.}
Consider the boundary of the solid parameter cylinder
$\overline{B}_\varepsilon(\lambda_0) \times [0,1]$, which consists of three pieces: the bottom face ($t=0$), the top face ($t=1$), and the lateral surface ($\lambda \in \partial B_\varepsilon$, $t \in [0,1]$). Define a continuous map on this boundary by assigning it to $A(\lambda)$ on the bottom, to $D(\lambda)$ on the top, and to $h_t(\lambda)$ on the lateral surface. Since $M_N(\bc)$ is a convex subset of a locally convex topological vector space, the straight-line convex combination
$A_t(\lambda) := (1-t)A(\lambda) + t D(\lambda)$
provides an explicit continuous extension to $\overline{B}_\varepsilon(\lambda_0) \times [0,1]$ which satisfies
\begin{enumerate}[label=(\alph*)]
  \item $A_0(\lambda) = A(\lambda)$ and $A_1(\lambda) = D(\lambda)$ for all
    $\lambda \in \overline{B}_\varepsilon(\lambda_0)$;
  \item $A_t(\lambda) = h_t(\lambda) \in GL_N(\bc)$ for all
    $\lambda \in \partial B_\varepsilon(\lambda_0)$ and $t \in [0,1]$.
\end{enumerate}
\vs
\noi\textbf{Step 3: Admissible $S^1$-equivariant deformation.}
Since every complex matrix commutes with the scalar $S^1$-action on $W \simeq N \cdot \mathcal U_m$, the family $A_t$ induces an $S^1$-equivariant deformation of the complemented operator:
\[
\mathcal A_t(\lambda,u) :=
  \bigl(\Theta(\lambda,u),\; A_t(\lambda)u\bigr).
\]
We verify $\mathscr O$-admissibility by checking each component of the boundary
$\partial \mathscr O$:
\begin{enumerate}[label=(\roman*)]
  \item \textit{Lateral boundary} ($\|u\| = \delta$):
the auxiliary function satisfies $\Theta(\lambda,u) > 0$, so $\mathcal A_t(\lambda,u) \neq 0$ regardless of $A_t(\lambda)$.
  \item \textit{Bounding caps} ($|\lambda - \lambda_0| = \varepsilon$, $\|u\| \leq \delta$): by property~(b), the matrix $A_t(\lambda) = h_t(\lambda) \in GL_N(\bc)$, so $A_t(\lambda)u = 0$ forces $u = 0$. At $u = 0$, the auxiliary function satisfies $\Theta(\lambda,0) < 0$, so $\mathcal A_t(\lambda,0) =(\Theta(\lambda,0), 0) \neq 0$.
\end{enumerate}
By the homotopy invariance property~\ref{s2} of the $S^1$-equivariant degree, one has
\begin{equation}\label{eq:det_red_homotopy}
S^1\text{-deg}\bigl(A_\Theta,\;
  \mathscr O\bigr)
= S^1\text{-deg}\bigl(D_\Theta,\; \mathscr O\bigr), \quad D_\Theta(\lambda) :=(\Theta, D(\lambda)).
\end{equation}
\vs
\noi\textbf{Step 4: Localization to the first coordinate.}
Writing $u = (u_1, u_2, \dots, u_N) \in \mathcal U_m^N$, the equation $D(\lambda)u = 0$ reads $d(\lambda)u_1 = 0$ and $u_k = 0$ for $k = 2,\dots,N$. In particular, every nontrivial zero ($u \neq 0$) of the complemented map $D_\Theta(\lambda)$ inside $\mathscr O$ lies in the $S^1$-invariant subspace
$\mathscr O_1 :=
\mathscr O \cap (\bc \times \mathcal U_m \times \{0\}^{N-1})$. By the additivity property~\ref{s1} of the $S^1$-equivariant degree, one has
\begin{equation}\label{eq:det_red_additivity}
\s1deg\bigl(D_\Theta,\;
  \mathscr O \bigr)
= \s1deg\bigl( D_1,\;
  \mathscr O_1\bigr), \quad D_1(\lambda) := (\Theta_1, d(\lambda)),
\end{equation}
where $\Theta_1(\lambda, u_1)
:=\Theta(\lambda, u_1, 0, \dots, 0)$ restricts the auxiliary function to the first coordinate subspace.
\vs
\noi\textbf{Step 5: Normalization.}
The admissible $S^1$-pair $(D_1,\; \mathscr O_1\bigr)$ is a complemented map on the single irreducible $S^1$-representation $\mathcal U_m$. Every nontrivial zero has isotropy exactly $\bz_m$, since $e^{i\theta}u_1 = e^{im\theta}u_1 = u_1$ if and only if
$\theta \in \frac{2\pi}{m}\bz$. By the normalization~\ref{s3} and additivity~\ref{s1} properties of the $S^1$-equivariant degree (applied to the finitely many $S^1$-orbits of zeros inside $\mathscr O_1$), the $S^1$-degree evaluates to
\[
S^1\text{-deg}\bigl( D_1,\;
  \mathscr O_1\bigr)
= \deg\bigl(d,\; B_\varepsilon(\lambda_0)\bigr) \cdot (\bz_m).
\]
Combining \eqref{eq:det_red_homotopy}, \eqref{eq:det_red_additivity}, and the above yields \eqref{eq:det_reduction}. The cumulative degree identity follows immediately by extracting the coefficient of $(\bz_m)$.
\end{proof}
\section{The Twisted Equivariant Degree} \label{sec:appendix_twisted_deg}
Let $\Gamma$ be a compact Lie group and define the product group $G:= S^1 \times \Gamma$. Any closed subgroup in $G$ can be uniquely identified with a triple $(K,\varphi,l)$, called the {\it twisted decomposition} of that subgroup, and consisting of a subgroup $K \leq \Gamma$, a homomorphism $\varphi: K \rightarrow S^1$ and a number $l \in \bn \cup \{0\}$ as follows
\begin{align} \label{def:twisted_subgroups}
 K^{\varphi,l} := \{(z,\gamma) \in S^1 \times K: \varphi(\gamma) = z^l \}.
\end{align}
\begin{remark} \rm
In the case that $l = 0$, the triple $(K,\varphi,0)$ is always (that is, for any subgroup $K \leq \Gamma$ and homomorphism $\varphi: K \rightarrow S^1$) associated with
the {\it product subgroup}
\[
K^{\varphi,0} = S^1 \times K.
\]
\end{remark}
We denote by $\Phi_1^t(G)$ the set of {\it conjugacy classes} of twisted subgroups \eqref{def:twisted_subgroups} with one-dimensional Weyl groups, i.e.
\begin{align} \label{def:twisted_subgroup_conjugacy_classes}
\Phi_1^t(G) := \{ (K^{\varphi,s}) \in \Phi(G) : \dim W(K^{\varphi,s}) =1 \},
\end{align}
and we define the free $\bz$-module generated by $\Phi_1^t(G)$ as follows
\[
A^t_1(G) := \bz[\Phi_1^t(G)].
\]
All elements $a \in A^t_1(G)$ are of the form
\begin{align}\label{eq:arbitrary_At1_element}
  a = \sum\limits_{(H) \in \Phi_1^t(G)} n_H(H),  \quad n_H \in \bz, 
\end{align}
where $n_H = 0$ except for a finite number of conjugacy classes $(H) \in  \Phi_1^t(G)$. The coefficient of any particular twisted orbit type $(H) \in \Phi_1^t(G)$ in \eqref{eq:arbitrary_At1_element} is specified with the notation
\begin{align} \label{def:coefficient_operator_notation}
\operatorname{coeff}^H(a) := n_H.   
\end{align}
There is a natural $A(\Gamma)$-module structure on $A^t_1(G)$ induced by the product $A(\Gamma) \times A^t_1(G) \rightarrow A^t_1(G)$ defined, for any pair of generators $((K),(H)) \in \Phi_0(\Gamma) \times \Phi_1^t(G)$, as follows
\begin{align} \label{def:module_product}
(K) \cdot (H^{\varphi,l}) := \sum\limits_{(L) \leq (K)} n_L (L^{\varphi,l}),
\end{align}
where $n_L \in \bz$ is given by the number of type $(L^{\varphi,l})$-type orbits in the $G$-space $\frac{G}{K \times S^1} \times \frac{G}{H^{\varphi,l}}$. More practically,
a multiplication table for the $A(\Gamma)$-module $A_1^t(G)$ can be computed using the recurrence formula
\begin{align} \label{def:recurrence_formula_module_product}
    n_L = \frac{n(L,K)|W_\Gamma(K)|n(L^{\varphi,l},H^{\varphi,l})|W(H^{\varphi,l})/S^1| - \sum_{(\tilde{L}) > (L)} n_{\tilde L} n(L^{\varphi,l},\tilde L^{\varphi,l}) | W(\tilde L^{\varphi,l})/S^1 |}{| W(L^{\varphi,l})/S^1 |}.
\end{align}
An {\it admissible $G$-pair} $(f, \Om)$ in an orthogonal $G$-representation $U$ consists of an open bounded $G$-invariant set $\Om \subset \br \times U$ and a continuous $G$-equivariant map  $f: \br \times U \rightarrow U$ with $f^{-1}(0) \cap \partial \Om = \emptyset$. We denote by $\mathcal M_1^{G}(U)$ the set of all admissible $G$-pairs in $U$ and by $\mathcal M_1^{G}$ the set of {\it all} admissible $G$-pairs, defined by taking a union over all orthogonal $G$-representations as follows 
\[
\mathcal M_1^{G} := \bigcup_{U} \mathcal M_1^{G}(U).
\]
Whereas an axiomatic definition for the $S^1$-degree was provided in Section \ref{sec:appendix_S1}, we now present a closed-form definition for the twisted $G$-equivariant degree.
\begin{definition} \rm
    The {\it twisted $G$-equivariant degree} is a map $\gdeg: \mathcal M_1^G \rightarrow A_1^t(G)$ assigning to every admissible $G$-pair $(f, \Om) \in \mathcal M_1^G$ the $A_1^t(G)$ element
        \[
\gdeg(f,\Om) := \sum\limits_{(H) \in \Phi_1^t(G)} n_H (H),
\]
where the integer coefficients $n_H \in \bz$ are given by the recurrence formula 
\begin{align} \label{def:twisted_degree_recurrence}
n_H := \frac{\deg_{S^1}(f^H,\Om^H) - \sum_{(L)>(H)} n_L n(H,L) | W(L)/S^1|}{|W(H)/S^1|}.    
\end{align}
\end{definition}
\begin{remark} \rm
The twisted $G$-equivariant degree satisfies the standard degree properties
\begin{enumerate}[label=($T_\arabic*$)]
\item\label{t1} {\bf (Additivity)} If, for any two $G$-invariant disjoint subsets $\Om_1,\Om_2 \subset \Om$, one has $f^{-1}(0) \cap \Om \subset \Om_1 \cup \Om_2$, then
\[
\gdeg(f,\Om) = \gdeg(f,\Om_1) + \gdeg(f,\Om_2).
\]
\item\label{t2} {\bf (Homotopy)} If $h: [0,1] \times \br \times V \rightarrow V$ is an $\Om$-admissible $G$-homotopy, then
\[
\gdeg(h_0,\Om) = \gdeg(h_1,\Om),
\]
where $h_t: \br \times V \rightarrow V$ is used to indicate the map $h_t(\cdot, \cdot) := h(t, \cdot,\cdot)$.
\item\label{t3}
{\bf (Normalization)}: Let $(f, \Omega) \in \mathcal{M}_1^{G}$ be such that $f$ is regular normal in $\Omega$. Assume further that $f^{-1}(0) \cap \Omega = G(w_0)$ consists of a single orbit for some $w_0 \in \Omega$. Then,
\begin{align*}
    \gdeg(f, \Omega) =
    \begin{cases}
    \rho_0 (G_{w_0}) & \text{if } (G_{w_0}) \in \Phi_1^t(G); \\
    0 & \text{otherwise},
    \end{cases}    
\end{align*}
where 
\begin{align}\label{def:orbit_index_ch2}
  \rho_0 := \operatorname{sign} \det(Df(w_0)|_{S_{w_0}})  
\end{align}
and $S_{w_0}$ is the positively oriented slice to the orbit $G(w_0)$ at $w_0$.
\end{enumerate}
\end{remark}
We require two additional properties, which can be derived from the axioms \ref{t1}--\ref{t3}, to effectively employ the twisted $G$-equivariant degree. The first, an existence property:
\begin{enumerate}[label=($T_\arabic*$), start=4]
    \item\label{t4} \textbf{(Existence)} If there is an orbit type $(H) \in \Phi_1^t(G)$ for which $\operatorname{coeff}^H(\gdeg(f,\Om)) \neq 0$, then there exists $x \in \Om$ such that $f(x) = 0$ and $(G_x) \geq (H)$;
\end{enumerate}
and the second, a product property, linked to the product property of the Brouwer equivariant degree:
\begin{lemma} \rm \label{lemm:twisted_degree_product_property}
For any admissible $\Gamma$-pair $(\phi,D) \in \mathcal M^\Gamma$ and for any admissible $G$-pair $(f,\Om) \in \mathcal M^{G}_1$, one has $(f \times \phi, \Om \times D) \in \mathcal M^{G}_1$ and
\[
\gdeg(f \times \phi, \Om \times D) = \Gamma \text{\rm -deg}(\phi,D) \cdot \gdeg(f,\Om).
\]
where $ \Gamma \text{\rm -deg}$ is the $\Gamma$-equivariant degree taking values in the Burnside ring $A(\Gamma)$.
\end{lemma}

\subsection{A Computational Formula for the Twisted $G$-Equivariant Leray-Schauder Degree} \label{sec:appendix_twisted_comp_form}
Let $V$ be a finite dimensional, orthogonal $\Gamma$-representation and let 
$\mathcal H$ be an isometric Banach $G$-representation of maps taking values in $V$. For each $m \in \bn$ we denote by $\mathcal U_m \simeq \bc$, the irreducible $S^1$-representation equipped with the {\it $m$-folding $S^1$-action}
\begin{align*} 
    e^{i \theta}w := e^{i m\theta} w, \quad \theta \in S^1, \; w \in \mathcal U_m,
\end{align*}
and by $\mathcal U_0 \simeq \br$, the irreducible $S^1$ representation on which $S^1$ acts trivially. Assuming that a complete list of irreducible $\Gamma$-representations $\{ \mathcal V_j \}_{j = 0}^r$ are made available, $\mathcal H$ has a
$G$-isotypic decomposition of the form
\begin{align*} 
    \mathcal H = \bigoplus_{m \geq 0} \bigoplus_{j = 0}^r  \mathcal H_{m,j},
\end{align*}
where each {\it $G$-isotypic component} $\mathcal H_{m,j}$  is modeled on the irreducible $G$-representation $\mathcal V_{m,j} := \mathcal U_m \otimes \mathcal V_j$.
\vs
Notice that, for $m = 0$, the irreducible $G$-representation $\mathcal V_{0,j}$ coincides with the irreducible $\Gamma$-representation $\mathcal V_j$. On the other hand, every irreducible $G$-representation $\mathcal V_{m,j}$ with $m > 0$ is associated with an admissible $G$-pair $(b_{m,j},\mathscr O_{m,j}) \in \mathcal M^{G}_1(\mathcal V_{m,j})$, called the {\it basic pair} on $\mathcal V_{m,j}$,  consisting of 
the $G$-invariant set,
\begin{align*}  
\mathscr O_{m,j} := \{ (t,v) \in \br \times \mathcal V_{m,j} : \frac{1}{2} < \|v\|_{\mathcal H} < 2, \; |t| < 1 \},
\end{align*}
and the $\mathscr O_{m,j}$-admissible $G$-equivariant map
\begin{align*} 
    b_{m,j}(t,v) := (1-\|v\|_{\mathcal H} + it) \cdot v.
\end{align*}
The {\it basic twisted degree} on $\mathcal V_{m,j}$, denoted $\deg_{\mathcal V_{m,j}} \in A_1^t(G)$, is defined as the twisted $G$-equivariant degree
\begin{align*} 
\deg_{\mathcal V_{m,j}} := \gdeg(b_{m,j}, \mathscr{O}_{m,j}).
\end{align*}
The following result can be derived from the recurrence formula \eqref{def:twisted_degree_recurrence}.
\begin{lemma} \rm
Given an irreducible  $G$-representation $\mathcal V_{m,j}$, the basic degree
    \[
    \deg_{\mathcal V_{m,j}} = \sum\limits_{(H) \in \Phi_1^t(G)} n_H (H),
    \]
is specified by the coefficients
    \[
    n_H := \frac{\frac{1}{2}\dim \mathcal V_{m,j}^H - \sum_{(L)>(H)} n_L n(H,L)|W(L)/S^1|}{|W(H)/S^1|}.
    \]
\end{lemma}
Let $a: S^1 \rightarrow GL^{G}(\mathcal H)$ be a continuous family of $G$-equivariant invertible linear operators. 
We are interested in a computational formula for the twisted $G$-equivariant Leray-Schauder degree of an admissible $G$-pair $(T,\mathscr D) \in \mathcal M^{G}_1(\mathcal H)$ consisting of the $G$-invariant set
\[
\mathscr D := \{ (\lambda,v) \in \bc \times \mathcal H : \| v \|_{\mathcal H} < 2, \; \frac{1}{2} < | \lambda | < 2 \},
\]
and the $\mathscr D$-admissible $G$-equivariant operator $T: \bc \times \mathcal H \rightarrow \br \times \mathcal H$ given by
\[
T(\lambda,v) := \left(\Theta(\lambda,v), a\left(\frac{\lambda}{|\lambda|}\right)v\right),
\]
where $\Theta: \overline{\mathscr D} \rightarrow \br$ is any $G$-invariant function satisfying
\begin{align*} 
\begin{cases}
\Theta(\lambda,0) < 0 \quad & \text{ for } |\lambda | = 2 \\
\Theta(\lambda,u) > 0  \quad & \text{ for } |\lambda | = \frac{1}{2}.
\end{cases}    
\end{align*}
\begin{remark}
Although admissible
$G$-pairs are formally defined on spaces of the form $\mathbb R\times \mathcal H \to \mathcal H$, a two-parameter cylinder $\mathbb C\times\mathcal H$ is obtained by absorbing one real parameter into
the representation.  Namely, put $\widehat{\mathcal H}:=\mathbb R\oplus\mathcal H$, where $G$ acts trivially on the additional $\mathbb R$-summand, such that $\mathbb R\times\widehat{\mathcal H} = \mathbb R\times(\mathbb R\oplus\mathcal H)
\simeq \mathbb C\times\mathcal H$.
Under this identification, $T$ is an $\mathscr D$-admissible $G$-map from $\mathbb R\times\widehat{\mathcal H}$ to $\widehat{\mathcal H}$. This is the standard formulation used in the $S^1$-equivariant and twisted equivariant degree treatments of Hopf
bifurcation, where the second parameter is typically the unknown frequency.
\end{remark}
For convenience, we adopt the notations $\mathscr D_{m,j} := \{(\lambda,v) \in \mathscr D : v \in \mathcal H_{m,j} \}$, $a_{m,j}:= a|_{\mathcal H_{m,j}} : \mathcal H_{m,j} \rightarrow \mathcal H_{m,j}$ and $T_{m,j}(\lambda,v) := (1 - |\lambda|, a_{m,j}(\lambda) \cdot v)$.
We are now in a position to present (see \cite{AED}, Chapter 4) the first of a set of two essential tools, the so-called Splitting Lemma, for this purpose.
\begin{lemma}\label{lemm:splitting_lemma} \rm
For any $m,m' > 0$, $j,j' \in \{0,1,\ldots,r\}$, one has
\[
\gdeg(T_{m,j} \times T_{m',j'},\mathscr D_{m,j} \times \mathscr D_{m',j'}) =  \gdeg(T_{m,j}, \mathscr D_{m,j}) + \gdeg(T_{m',j'}, \mathscr D_{m',j'})
\]
\end{lemma}
The second tool essential to the computation of $\gdeg(T,\mathscr D)$ provides a means to relate the twisted $G$-equivariant Leray-Schauder degrees $\gdeg(T_{m,j}, \mathscr D_{m,j}) \in A_1^t(G)$ with the corresponding basic degree $\deg_{\mathcal V_{m,j}}$.
\begin{lemma} \rm
For any $m > 0$, $j \in \{0,1,\ldots,r\}$, one has
    \[
    \gdeg( T_{m,j}, \mathscr D_{m,j}) = \deg(\det\nolimits_\bc(a_{m,j})) \deg_{\mathcal V_{m,j}}.
    \]
\end{lemma}
Combining the product property of the degree with the splitting lemma gives the standard basic-degree factorization used in the twisted equivariant degree computations.
\begin{corollary}
\label{cor:standard_basic_degree_factorization}
Let $(\lambda_0,0)$ be an isolated dynamic critical point. Suppose the stationary part contributes negative stationary multiplicities $m_j^-(\lambda_0)$ and the non-stationary blocks have winding numbers $\rho_{m,j}(\lambda_0)$. Then the local bifurcation invariant factors as
\[
\omega_G(\lambda_0)
    = \prod_{j=0}^{r}
    \left(\deg_{\mathcal V_j}\right)^{m_j^-(\lambda_0)}
    \cdot
    \sum_{m>0}\sum_{j=0}^{r}
    \rho_{m,j}(\lambda_0)\deg_{\mathcal V_{m,j}}.
\]
\end{corollary}


\begin{thebibliography}{00}

\bibitem{AtiyahPatodiSinger}
M.~F. Atiyah, V.~K. Patodi, and I.~M. Singer,
\emph{Spectral asymmetry and Riemannian geometry. III},
Math. Proc. Cambridge Philos. Soc., \textbf{79} (1976), 71--99.

\bibitem{AED}
Z.~Balanov, W.~Krawcewicz, and H.~Steinlein,
\emph{Applied Equivariant Degree},
AIMS Series on Differential Equations \& Dynamical Systems, Vol.~1, 2006.

\bibitem{book-new} Z. Balanov, W. Krawcewicz,  D. Rachinskii, J. Yu, H-P. Wu, \emph{Degree Theory and Symmetric Equations Assisted by GAP System: With a Special Focus on Systems with Hysteresis}, AMS Mathematical Surveys and Monographs, Vol. 286 (2025).

\bibitem{survey}
Z.~Balanov, W.~Krawcewicz, S.~Rybicki, and H.~Steinlein, \emph{A short treatise on the equivariant degree theory and its applications},
J. Fixed Point Theory Appl., \textbf{8} (2010), 1--74.

\bibitem{BalanovKrawcewiczRuan2006}
Z.~Balanov, W.~Krawcewicz, and H.~Ruan,
\emph{Hopf bifurcation in a symmetric configuration of lossless transmission lines},
Nonlinear Anal., \textbf{64} (2006), 1144--1166.

\bibitem{Dylawerski1991}
G.~Dylawerski, K.~G\c{e}ba, J.~Jodel, and W.~Marzantowicz, \emph{An $S^1$-equivariant degree and the Fuller index},
Ann. Polon. Math., \textbf{52} (1991), 243--280.

\bibitem{Dylawerski}
G.~Dylawerski,
\emph{An $S^1$-degree and $S^1$-maps between representation spheres}, Equivariant Homotopy and Cohomology Theory, Contemp. Math., Amer. Math. Soc., 1988.

\bibitem{Erbe1992}
L.~H. Erbe, W.~Krawcewicz, and J.~Wu,
\emph{Leray--Schauder degree for semilinear Fredholm maps and periodic boundary value problems of neutral equations}, Nonlinear Anal., \textbf{15} (1990), 747--764.

\bibitem{Fuller1967}
F.~B. Fuller,
\emph{An index of fixed point type for periodic orbits}, Amer. J. Math., \textbf{89} (1967), 133--148.

\bibitem{FuraRatajczakRybicki}
H.~Fura, A.~Ratajczak, and S.~Rybicki,
\emph{Existence and continuation of periodic solutions of autonomous Newtonian systems},
J. Differential Equations, \textbf{218} (2005), 216--252.

\bibitem{GGK_NODEA}
C. Garcia-Azpeitia, Z. Ghanem, and W. Krawcewicz,
\emph{Global Bifurcation of Spiral Wave Solutions to the Complex Ginzburg-Landau Equation}, Nonlinear Differential Equations and Applications, (2025).

\bibitem{GaGhKr}
\emph{C.~Garc\'{\i}a-Azpeitia, Z.~Ghanem, and W.~Krawcewicz, Global bifurcation in symmetric systems of nonlinear wave equations},
J. Differential Equations, (2025).

\bibitem{Geba1994}
K.~G\c{e}ba, W.~Krawcewicz, and J.~Wu,
\emph{An equivariant degree with applications to symmetric bifurcation problems, Part~1: Construction of the degree}, Proc. London Math. Soc. \textbf{69} (1994), 377--398.

\bibitem{Ize1992}
J.~Ize, I.~Massab\`o, and A.~Vignoli,
\emph{Degree theory for equivariant maps, the general $S^1$-action}, Mem. Amer. Math. Soc., \textbf{100} (1992), no.~481.

\bibitem{Ize1989}
J.~Ize, I.~Massab\`o, and A.~Vignoli,
\emph{Degree theory for equivariant maps}, I,
Trans. Amer. Math. Soc., \textbf{315} (1989), 433--510.

\bibitem{IhrigGolubitsky1984}
E.~Ihrig and M.~Golubitsky,
\emph{Pattern selection with $O(3)$ symmetry},
Physica D: Nonlinear Phenomena, \textbf{13}(1-2) (1984), 1--33.

\bibitem{Izydorek2021}
M. Izydorek, J. Janczewska, and N. Waterstraat, \emph{The equivariant spectral flow and bifurcation of periodic solutions of Hamiltonian systems}, Nonlinear Analysis, Vol. 211 (2021), 112475.

\bibitem{Krasnoselskii}
M.~A. Krasnosel'skii,
\emph{Topological Methods in the Theory of Nonlinear Integral Equations},
Pergamon Press, Oxford, 1964.

\bibitem{KrawcewiczWuBook}
W.~Krawcewicz and J.~Wu,
\emph{Theory of Degrees with Applications to Bifurcations and Differential Equations},
CMS Series of Monographs, Wiley, New York, 1997.

\bibitem{LeraySchauder}
J.~Leray and J.~Schauder,
\emph{Topologie et \'equations fonctionnelles},
Ann. Sci. \'Ecole Norm. Sup., \textbf{51} (1934), 45--78.


\bibitem{Rabinowitz}
P.~H. Rabinowitz,
\emph{Some global results for nonlinear eigenvalue problems},
J. Funct. Anal. \textbf{7} (1971), 487--513.

\bibitem{GAP} H-P. Wu, GAP system package EquiDeg for the computations of the Burnside ring and the equivariant
degree, https://github.com/psistwu/equideg.
\end{thebibliography}
\end{document}